\documentclass[12pt]{amsart}
\usepackage[margin=1in]{geometry}

\usepackage{amssymb}
\usepackage{mathtools}
\usepackage{enumitem}
\usepackage{array}
\usepackage{hyperref}
\usepackage{microtype}

\newcommand\D{\mathcal D}
\newcommand\R{\mathbb R}
\newcommand\Z{\mathbb Z}
\newcommand\tint{\begingroup\textstyle\int\endgroup}
\newcommand\tdashint{\begingroup\textstyle\dashint\endgroup}
\newcommand\tesssup{\begingroup\textstyle\esssup\endgroup}

\DeclareMathOperator*{\esssup}{ess\,sup}
\DeclareMathOperator*{\essinf}{ess\,inf}
\DeclareMathOperator*{\BMO}{BMO}

\DeclareMathOperator*{\Mu}{Mu}

\def\XXint#1#2#3{{\setbox0=\hbox{$#1{#2#3}{\int}$}
		\vcenter{\hbox{$#2#3$}}\kern-.5\wd0}}
\def\Xint#1{\mathchoice
	{\XXint\displaystyle\textstyle{#1}}
	{\XXint\textstyle\scriptstyle{#1}}
	{\XXint\scriptstyle\scriptscriptstyle{#1}}
	{\XXint\scriptscriptstyle\scriptscriptstyle{#1}}
	\!\int}
\def\dashint{\Xint-}

\theoremstyle{plain}
\newtheorem{theorem}{Theorem}[section]
\newtheorem{proposition}[theorem]{Proposition}
\newtheorem{lemma}[theorem]{Lemma}
\newtheorem{corollary}[theorem]{Corollary}

\theoremstyle{definition}
\newtheorem{definition}[theorem]{Definition}
\newtheorem{example}[theorem]{Example}

\theoremstyle{remark}
\newtheorem{remark}[theorem]{Remark}

\numberwithin{equation}{section}

\begin{document}

\title[One-sided median porosity and distance functions]{One-sided median porous sets and one-sided Muckenhoupt distance functions}

\author[A. C. Goksan]{Alptekin Can Goksan}
\address{Department of Mathematics, University of Toronto, Toronto, Ontario, Canada}
\email{a.goksan@mail.utoronto.ca}

\author[I. Uriarte-Tuero]{Ignacio Uriarte-Tuero}
\address{Department of Mathematics, University of Toronto, Toronto, Ontario, Canada}
\email{ignacio.uriartetuero@utoronto.ca}


\begin{abstract}

We introduce the notion of one-sided median porosity for subsets $E$ of $\R$. We prove that this condition is necessary and sufficient for the distance weight $d_E^{\,-\alpha}$ to belong to a one-sided Muckenhoupt $A_p$ class for some $\alpha>0$ and $1<p<\infty$. As part of the proof, we obtain new characterizations of one-sided $A_p$ weights and one-sided $\BMO$ functions, in terms of medians. It was recently shown that $d_E^{\,-\alpha}$ is a one-sided Muckenhoupt $A_1$ weight for some $\alpha>0$ if and only if $E$ is one-sided weakly porous. In this paper, we find the precise range of exponents $\alpha>0$ such that $d_E^{\,-\alpha}$ belongs to a one-sided $A_p$ class, both for $p=1$ and for $1<p<\infty$. In addition, we show that $E$ is median porous if and only if it is both left and right median porous, and we give an example of a one-sided median porous set which is neither median porous nor one-sided weakly porous.

\end{abstract}

\maketitle

\section{Introduction}
\label{sec_intro}

Muckenhoupt $A_p$ weights are one of the main objects of study in harmonic analysis and have many applications in PDE. Their importance stems from the fact that, for any $1 < p < \infty$, singular integrals and maximal operators are bounded on $L^p(w)$ if and only if the weight $w$ belongs to the $A_p$ class (see \cite{GR} and the references therein). In recent years, there has been interest in weights of the form
\[
w_{E,\alpha}(x) = d_E(x)^{\,-\alpha}, \qquad x \in \R^n,
\]
where $E \subset \R^n$ is a nonempty set, $d_E(x)$ is the distance from $x$ to $E$, and $\alpha>0$. We refer to such weighs as \emph{distance weights}. The goal is to find geometric properties of the set $E$ which are necessary and sufficient for $w_{E,\alpha}$ to belong to an $A_p$ class. Such characterizations have important implications for a number of topics in harmonic analysis and PDE, such as quasiadditivity of Riesz capacities \cite{A}, regularity of solutions to various classes of PDE \cite{H,ACDT,DL}, and a geometric description of domains that support Hardy-Sobolev inequalities \cite{KLV,LV,DILTV}.

One of the first results on distance weights was that, if $E$ is porous, then $w_{E,\alpha} \in A_1$ if and only if $0 \le \alpha < n - \dim_A(E)$, where $\dim_A(E)$ is the Assouad dimension of $E$ \cite{DILTV}. Recall that a set $E$ is said to be \emph{porous} if every ball in $\R^n$ of sufficiently small radius contains a ball of comparable radius which is \emph{$E$-free} (i.e.\ disjoint from $E$). (See \cite{S2} for a survey of porosity and its connections to various notions of dimension in geometric measure theory.) Earlier, it had been shown that, if $E \subset \mathbb{T}$, where $\mathbb{T}$ is the unit circle in $\R^2$, then $w_{E,\alpha} \in A_1(\mathbb{T})$ for some $\alpha>0$ if and only if $E$ satisfies a condition called \emph{weak porosity} \cite{V}. Recently, the definition of weak porosity was extended to arbitrary subsets of $\R^n$ by Anderson, Lehrb\"ack, Mudarra and V\"ah\"akangas \cite{ALMV}, who subsequently proved the following elegant theorem.

\begin{theorem}
	\label{thm_ALMV1}
	Let $\emptyset \ne E \subset \R^n$. Then the following are equivalent:
	\begin{enumerate}[label=\upshape(\alph*)]
		\item There exists $\alpha>0$ such that $d_E^{\,-\alpha} \in A_1$.
		\item $E$ is $s$-weakly porous for all $0<s<1$.
		\item $E$ is $s$-weakly porous for some $0<s<1$.
	\end{enumerate}
\end{theorem}

Roughly speaking, a set $E$ is weakly porous if every cube $Q$ in $\R^n$ contains pairwise disjoint $E$-free subcubes whose side lengths are bounded below by a constant times the side length of a maximal $E$-free subcube of $Q$ and whose volumes add up to at least a constant times the volume of $Q$. (See Definition \ref{def_porous}(a) and the paragraph following Definition \ref{def_porous} for the precise definition in the case $n=1$.) Thus, Theorem \ref{thm_ALMV1} relates the $A_1$ property of the weight $w_{E,\alpha}$ to geometric properties of the set $E$. It was later shown that Theorem \ref{thm_ALMV1} can be extended to metric spaces equipped with doubling measures \cite{M} and, more generally, to spaces of homogeneous type \cite{AGG}.

Various papers that addressed the case $p=1$ contained partial results for the case $1<p<\infty$, which became a salient question. Using a novel characterization of the $\BMO$ space, Pasquariello and the second named author obtained a geometric characterization of the sets $E$ for which $w_{E,\alpha} \in A_p$ for some $\alpha>0$ and $1<p<\infty$ \cite{PU}. They introduced the notion of \emph{median porosity} and proved the following theorem.

\begin{theorem}
	\label{thm_PU1}
	Let $\emptyset \ne E \subset \R^n$. Then the following are equivalent:
	\begin{enumerate}[label=\upshape(\alph*)]
		\item There exist $\alpha>0$ and $1<p<\infty$ such that $d_E^{\,-\alpha} \in A_p$.
		\item $E$ is $(s,t)$-median porous for all $0<s<t<1$.
		\item $E$ is $(s,t)$-median porous for some $0<s<t<1$.
	\end{enumerate}
\end{theorem}

Median porosity is strictly weaker than weak porosity. The philosophical reason for this is that the side length of a maximal $E$-free subcube of $Q$ is replaced with a smaller quantity, which is still related to the side lengths of $E$-free subcubes of $Q$. (See Definition \ref{def_porous}(b) and the paragraph following Definition \ref{def_porous} for the precise definition in the case $n=1$.) Moreover, median porosity is considerably harder to manipulate than weak porosity, which is unsurprising since $A_p$ is less restrictive than $A_1$. Since it was introduced, the concept of median porosity has been the subject of further research. For example, it was recently shown that median porosity is quasiconformally invariant, whereas weak porosity is not \cite{KV}.

It should be noted that a characterization similar to Theorem \ref{thm_PU1} was obtained independently by G\'omez Vargas \cite{G}.

Theorems \ref{thm_ALMV1} and \ref{thm_PU1} give necessary and sufficient conditions for $w_{E,\alpha}$ to be an $A_p$ weight for some $\alpha>0$ and $1 \le p < \infty$. In the one-dimensional case (where cubes $Q$ reduce to intervals $I$), there are also one-sided analogues of the $A_p$ classes, called $A_p^+$ and $A_p^-$ (see Definition \ref{def_Ap}), which characterize $L^p$ boundedness of one-sided singular integrals and maximal operators. A natural next step in the theory of distance weights is to look for necessary and sufficient conditions for $w_{E,\alpha}$ to be an $A_p^+$ weight for some $\alpha>0$ and $1 \le p < \infty$. For $p=1$, this problem was solved by Aimar, G\'omez, G\'omez Vargas and Mart\'in-Reyes \cite{AGGM}. They formulated a one-sided analogue of weak porosity, which they called \emph{right weak porosity}, and established the following nice result.

\begin{theorem}
	\label{thm_AGGM}
	Let $\emptyset \ne E \subset \R$. Then the following are equivalent:
	\begin{enumerate}[label=\upshape(\alph*)]
		\item There exists $\alpha>0$ such that $d_E^{\,-\alpha} \in A_1^+$.
		\item $E$ is $s$-right weakly porous for all $0<s<1$.
		\item $E$ is $s$-right weakly porous for some $0<s<1$.
	\end{enumerate}
\end{theorem}

The definition of right weak porosity is similar to that of weak porosity, but it uses information about pore sizes in the right half $I^+$ of an interval $I$ to deduce information about pore sizes in the left half $I^-$ of the same interval $I$ (see Definition \ref{def_porous}(c) and the paragraph following Definition \ref{def_porous}). It is worth mentioning that Theorem \ref{thm_AGGM} has been generalized to higher dimensions \cite{LMV} in the context of parabolic $A_1$ weights, which were introduced (together with parabolic $A_p$ weights, where $1<p<\infty$) in \cite{KS} as a versatile tool for the study of doubly nonlinear parabolic PDE.

Theorem \ref{thm_AGGM} showed that the theory of distance weights extends nicely to the one-sided case when $p=1$. However, it remained unknown if this is also the case when $1 < p < \infty$.

In this paper, we develop the theory of distance weights in the one-sided case for $1 < p < \infty$. We define a one-sided analogue of median porosity, which we call \emph{right median porosity} (see Definition \ref{def_porous}(d) and the paragraph following Definition \ref{def_porous}), and we show that this property is necessary and sufficient for $w_{E,\alpha}$ to be an $A_p^+$ weight for some $\alpha>0$ and $1 < p < \infty$. To be precise, we prove the following.

\begin{theorem}
	\label{thm_main}
	Let $\emptyset \ne E \subset \R$. Then the following are equivalent:
	\begin{enumerate}[label=\upshape(\alph*)]
		\item There exist $1<p<\infty$ and $\alpha>0$ such that $d_E^{\,-\alpha} \in A_p^+$.
		\item For every $1<p<\infty$, there exists $\alpha>0$ such that $d_E^{\,-\alpha} \in A_p^+$.
		\item $E$ is $(s,t)$-right median porous for some $0<s<t<1$.
		\item $E$ is $(s,t)$-right median porous for all $0<s<t<1$.
	\end{enumerate}
\end{theorem}

In the course of the proof of Theorem \ref{thm_main}, we obtain a number of noteworthy results which are of independent interest. In these results, the concept of the $s$-median $(0<s<1)$ of a function $f$ over an interval $I$ plays a central role. The first result is a characterization of $A_p^+$ in terms of medians, inspired by a similar characterization of $A_p$ due to Str\"omberg and Torchinsky \cite {ST}. Our characterization is as follows. (See the beginning of Section \ref{sec_prelim} for the definition of $I^{\gamma,\pm}$, which is a generalization of $I^\pm$.)

\begin{theorem}
	\label{thm_Ap}
	Let $w$ be a weight on $\R$. Then the following are equivalent:
	\begin{enumerate}[label=\upshape(\alph*)]
		\item $w \in A_p^+$ for some $1<p<\infty$.
		\item For every $0 < \gamma \le \frac{1}{2}$ and $0<s<1$, there exists $C>0$ such that, for all $I\subset\R$,
		\begin{equation}
			\label{eq_Ap2}
			\dashint_{I^{\gamma,-}} w \le C M_s(w, I^{\gamma,+}).
		\end{equation}
		\item There exist $0 < \gamma \le \frac{1}{2}$, $0<s<1$ and $C>0$ such that \eqref{eq_Ap2} holds for all $I\subset\R$.
	\end{enumerate}
\end{theorem}

The second major byproduct of our work is a new characterization of the one-sided $\BMO$ space $\BMO^+$ (see Definition \ref{def_BMO}) in terms of median differences. We draw inspiration from a similar characterization of standard $\BMO$ due to Pasquariello and the second named author \cite{PU}, but we use entirely different techniques in the proof. Our characterization is as follows.

\begin{theorem}
	\label{thm_BMO}
	Let $f : \R \to \R$ be measurable. Then the following are equivalent:
	\begin{enumerate}[label=\upshape(\alph*)]
		\item $f \in \BMO^+$.
		\item For every $0<s<t<1$, there exists $C>0$ such that, for all $I\subset\R$, we have
		\begin{equation}
			\label{eq_BMO2}
			M_t(f,I^-) - M_s(f,I^+) \le C.
		\end{equation}
		\item There exist $0<s<t<1$ and $C>0$ such that \eqref{eq_BMO2} holds for all $I\subset\R$.
	\end{enumerate}
\end{theorem}

Let us give a quick overview of the proof of Theorem \ref{thm_main}. There are a number of major obstacles involved in the proof. First, the technique used in \cite{AGGM} to prove Theorem \ref{thm_AGGM} (which deals with the case $p=1$) relies heavily on the observation that the right half $I^+$ of a certain interval $I$ contains an $E$-free interval whose size is bounded below by a certain number. In the $1<p<\infty$ case, the existence of one such interval is not enough; instead, a certain proportion of $I^+$ must be covered by such intervals, and this is not necessarily the case. Thus, the argument in \cite{AGGM} breaks down completely when $1<p<\infty$, and a new approach is needed. Second, in \cite{PU}, Theorem \ref{thm_PU1} (which deals with standard $A_p$) is proved by first establishing a sparse bound for arbitrary measurable functions in terms of median differences and then using this sparse bound to control the oscillations of $\log w_{E,\alpha}$. Unfortunately, the stopping time argument that is used to prove the sparse bound cannot easily be adapted to the one-sided case, so this approach does not work either.

We circumvent these difficulties by first proving a John-Nirenberg inequality for measurable functions that satisfy a bound of the form \eqref{eq_BMO2} (see Theorem \ref{thm_JN}). We then use this John-Nirenberg inequality and Theorem \ref{thm_Ap} to prove Theorem \ref{thm_BMO}, which is a key ingredient in the proof of Theorem \ref{thm_main}. Our proof of Theorem \ref{thm_JN} is inspired by the work of Mart\'in-Reyes and de la Torre \cite{MT}, who prove a similar John-Nirenberg inequality for functions in $\BMO^+$. Since we are assuming that the function satisfies inequality \eqref{eq_BMO2}, which involves medians, and not that the function has the $\BMO^+$ property, which involves averages, the proof is harder, and a number of new ideas are required. One of our key observations is that minimal and maximal medians have opposite continuity properties (see Lemmas \ref{lem_median3} and \ref{lem_median4}). We take advantage of these properties by using the minimal median instead of the more familiar maximal median at several judiciously chosen steps. Another key observation is that \eqref{eq_BMO2} implies a similar bound in which $I^-$ is replaced with a dyadic subinterval of $I^-$ (see Lemma \ref{lem_BMO}). This allows us to control the lengths of subintervals of $I^-$ consisting of points for which a median exceeds a certain value, which is a crucial step in the proof.

As a consequence of Theorems \ref{thm_PU1} and \ref{thm_main} and the fact that $A_p = A_p^+ \cap A_p^-$, we deduce that a set $E$ is median porous if and only if $E$ is both left median porous and right median porous (see Proposition \ref{prop_conseq1}). Thus, the theory of one-sided median porous sets relates to the theory of standard median porous sets in a nice and simple way.

Based on well-known inclusions between the $A_p$ classes, one would expect right median porosity to be the weakest of the four notions of porosity considered in Theorems \ref{thm_ALMV1}, \ref{thm_PU1}, \ref{thm_AGGM} and \ref{thm_main}. Indeed, this turns out to be the case. In Examples \ref{ex1}, \ref{ex2} and \ref{ex3}, we exhibit sets $E$ which are right median porous but do not satisfy some or any of the other three notions of porosity. Thus, right median porosity is strictly weaker than those other conditions.

Once the equivalences in Theorems \ref{thm_ALMV1}, \ref{thm_PU1}, \ref{thm_AGGM} and \ref{thm_main} have been established, it is natural to ask the following question: For what values of $\alpha$ does $w_{E,\alpha}$ belong to some $A_p$ class? For $p = 1$, this question was answered by Anderson, Lehrb\"ack, Mudarra and V\"ah\"akangas \cite{ALMV}. They defined a quantity called the $1$-Muckenhoupt exponent of $E$ (see Definition \ref{def_exp1}(a)), which is denoted by $\Mu_1(E)$ and is similar to the Assouad dimension, and proved the following.

\begin{theorem}
	\label{thm_ALMV2}
	Let $\emptyset \ne E \subset \R^n$ and let $\alpha > 0$. Then $d_E^{\,-\alpha} \in A_1$ if and only if $\alpha < \Mu_1(E)$.
\end{theorem}

For $1 < p < \infty$, the precise range of exponents $\alpha$ was found by Pasquariello and the second named author \cite{PU}. They defined the $\infty$-Muckenhoupt exponent of $E$ (see Definition \ref{def_exp1}(b)), which is denoted by $\Mu_\infty(E)$, and proved the following.

\begin{theorem}
	\label{thm_PU2}
	Let $\emptyset \ne E \subset \R^n$ and let $\alpha > 0$. Then $d_E^{\,-\alpha} \in A_p$ for some $1<p<\infty$ if and only if $\alpha < \Mu_\infty(E)$.
\end{theorem}

It is desirable to extend the quantitative characterizations in Theorems \ref{thm_ALMV2} and \ref{thm_PU2} to the one-sided case. Our last major result in this paper achieves precisely that goal. We define one-sided analogues $\Mu_1^+(E)$ and $\Mu_\infty^+(E)$ of the Muckenhoupt exponents of $E$ (see Definition \ref{def_exp2}) and prove the following two theorems.

\begin{theorem}
	\label{thm_exp1}
	Let $\emptyset \ne E \subset \R$ and $\alpha>0$. Then $d_E^{\,-\alpha} \in A_1^+$ if and only if $\alpha < \Mu_1^+(E)$.
\end{theorem}

\begin{theorem}
	\label{thm_exp2}
	Let $\emptyset \ne E \subset \R$ and $\alpha>0$. Then $d_E^{\,-\alpha} \in A_p^+$ for some $1<p<\infty$ if and only if $\alpha < \Mu_\infty^+(E)$.
\end{theorem}

The proofs of Theorems \ref{thm_exp1} and \ref{thm_exp2} present unique challenges, especially in the ``if'' direction. One of these is that, in the one-sided case, it is harder to prove that $w_{E,\alpha}$ is locally integrable. Indeed, we obtain integrability only on intervals with right endpoint in $E$, and we may not assume $E$ is a closed set. This is an issue in the case $\sup E<\infty$, and we overcome it by finding an increasing sequence of intervals on which the integrals of $w_{E,\alpha}$ are uniformly bounded and taking a limit. Another difficulty is that, in the one-sided case, after an integral inequality has been proved for intervals centred at points of $E$, it is not as easy to extend this inequality to arbitrary intervals, again due to the one-sided nature of the inequality. We deal with this issue by considering a number of cases, splitting certain integrals into two, and using a different set of estimates in each case to handle the terms appropriately.

By combining the results in \cite{ALMV}, \cite{PU}, \cite{AGGM} and the present paper, we obtain the following summary of the various notions of porosity for arbitrary nonempty sets $E \subset \R$.

\medskip
{\renewcommand{\arraystretch}{1.25}
\begin{center}
	\footnotesize
	\begin{tabular}{
			|>{\centering\arraybackslash}m{4em}
			|@{\hspace{2pt}}
			|>{\centering\arraybackslash}m{10em}
			|>{\centering\arraybackslash}m{8em}
			|>{\centering\arraybackslash}m{7em}
			|>{\centering\arraybackslash}m{10em}|}
		\hline
		& Definition & Distance function characterization & Dimension characterization & Example \\
		\hline\hline
		Weakly porous & $|\bar{E}|=0$ and $L_s(E,I) \gtrsim L_1(E,I)$ for some $0<s<1$ & $d_E^{\,-\alpha} \in A_1$ \quad \quad \quad for some $\alpha>0$ & $\Mu_1(E)>0$ & $\Z$ \\
		\hline
		Median porous & $|\bar{E}|=0$ and $L_s(E,I) \gtrsim L_t(E,I)$ \quad for some $0<s<t<1$ & $d_E^{\,-\alpha} \in A_p$ \quad \quad \quad for some $\alpha>0$ and $1<p<\infty$ & $\Mu_\infty(E)>0$ & $E_\gamma = \{\pm n^\gamma : n \in \Z_{\ge0}\}$ for $0<\gamma<1$ \\
		\hline
		Right weakly porous & $|\bar{E}|=0$ and $L_s(E,I^-) \gtrsim L_1(E,I^+)$ for some $0<s<1$ & $d_E^{\,-\alpha} \in A_1^+$ \quad \quad \quad for some $\alpha>0$ & $\Mu_1^+(E)>0$ & $\Z_{\ge0}$ \\
		\hline
		Right median porous & $|\bar{E}|=0$ and $L_s(E,I^-) \gtrsim L_t(E,I^+)$ for some $0<s<t<1$ & $d_E^{\,-\alpha} \in A_p^+$ \quad \quad \quad for some $\alpha>0$ and $1<p<\infty$ & $\Mu_\infty^+(E)>0$ & $(E_\gamma \cap (-\infty,0]) \cup (E_{\gamma/m} \cap [0,+\infty))$ for $0<\gamma<1$ and $m \ge 2$ \\
		\hline
	\end{tabular}
\end{center}}
\medskip

This paper is organized as follows: In Section \ref{sec_prelim}, we introduce notation, recall various definitions, and prove a few preliminary results. In Section \ref{sec_ApBMO}, we prove Theorems \ref{thm_Ap} and \ref{thm_BMO}, which characterize $A_p^+$ and $\BMO^+$, respectively, in terms of medians. In Section \ref{sec_main}, we prove Theorem \ref{thm_main}, which characterizes the sets $E$ such that the weight $w_{E,\alpha}$ belongs to $A_p^+$ for some $\alpha>0$ and $1<p<\infty$. In Section \ref{sec_conseq}, we look at consequences of Theorem \ref{thm_main} and present examples of sets which are right median porous but do not satisfy other notions of porosity. In Section \ref{sec_exp}, we prove Theorems \ref{thm_exp1} and \ref{thm_exp2}, which establish the precise range of exponents $\alpha>0$ such that $w_{E,\alpha}$ belongs to $A_1^+$ or to $A_p^+$ for some $1<p<\infty$. In Section \ref{sec_extra}, we include a few further results on right median porosity for the interested reader.

\section{Preliminaries}
\label{sec_prelim}

\subsection{Intervals and distance functions}

Our setting is the real line $\R$ equipped with Lebesgue measure. We write $I\subset\R$ to indicate that $I$ is an open interval $(a,b)$ of $\R$, and we write $J \subset I$ to indicate that $J$ is an open subinterval $(c,d)$ of $I$. (This convention applies only to the letters $I$ and $J$.) In what follows, assume $I = (a,b)$. We denote by $I^-$ and $I^+$ the left and right halves of $I$, respectively, i.e.
\[
I^- = (a,\tfrac{a+b}{2}), \qquad
I^+ = (\tfrac{a+b}{2},b).
\]
We denote by $I^l$, $I^c$ and $I^r$ the left, center, and right thirds of $I$, respectively, i.e.
\[
I^l = (a, \tfrac{2a+b}{3}), \qquad
I^c = (\tfrac{2a+b}{3}, \tfrac{a+2b}{3}), \qquad
I^r = (\tfrac{a+2b}{3}, b).
\]
More generally, for any $0 < \gamma \le \frac{1}{2}$, we define
\[
I^{\gamma,-} = (a, a + \gamma(b-a)), \qquad
I^{\gamma,+} = (b - \gamma(b-a), b).
\]
For example, when $\gamma = \frac{1}{2}$ we have $I^{\gamma,-} = I^-$ and $I^{\gamma,+} = I^+$, and when $\gamma = \frac{1}{3}$ we have $I^{\gamma,-} = I^l$ and $I^{\gamma,+} = I^r$.

For any $I\subset\R$ and any $c>0$, we denote by $cI$ the open interval with the same midpoint as $I$ and $c$ times the length. When $I = (a,b)$ and $J = (b,c)$ are two consecutive open intervals, we sometimes denote by $I \cup J$ the whole interval $(a,c)$, including the point $b$. This also generalizes to any finite or countable collection of consecutive open intervals.

We define the dyadic descendants of an open interval $I$ as follows: Let $\D_0(I) = \{I\}$ and, for $n \ge 1$, let $\D_n(I) = \{J^-,J^+ : J \in \D_{n-1}(I)\}$. Note that $\D_n(I)$ is the set of $n$th-generation dyadic descendants of $I$. The set of all dyadic descendants of $I$ is $\D(I) = \bigcup_{n=0}^\infty \D_n(I)$.

For any locally integrable function $f : \R \to \R$ and any $I\subset\R$, we denote the average value of $f$ over $I$ by
\[
f_I = \dashint_I f = \frac{1}{|I|} \int_I f(x)\,dx.
\]
For any two nonnegative functions $f$ and $g$ and any parameter $\alpha$, we use the notation $f \lesssim_\alpha g$ to indicate that $f \le Cg$ for some constant $C$ depending only on $\alpha$. We write $f \approx_\alpha g$ when both $f \lesssim_\alpha g$ and $g \lesssim_\alpha f$.

Let $E \subset \R$ be a nonempty set. For any point $x \in \R$, we denote the distance from $x$ to $E$ by
\[
d_E (x) = \inf\{|x-y|  : y \in E\}.
\]
If $F \subset \R$ is another nonempty set, we denote the distance between $E$ and $F$ by
\[
d(E,F) = \inf\{|x-y| : x \in E, y \in F\}.
\]
For any $r>0$, we define the $r$-neighbourhood of $E$ by
\[
E_r = \{x \in \R : d_E(x) < r\}.
\]
Recall that $\bar{E}$, the closure of $E$, satisfies
\[
\bar{E} = \{x \in \R : d_E(x)=0\}.
\]

\subsection{\texorpdfstring{$A_p$}{A\_p} weights and \texorpdfstring{$\BMO$}{BMO}}

By a \emph{weight on} $\R$, we mean a function $w : \R \to [0,\infty]$ such that $w$ is locally integrable (in particular, $w < \infty$ a.e.) and $w>0$ a.e. Of particular importance in harmonic analysis are the Muckenhoupt $A_p$ weights, which are defined as follows.

\begin{definition}
	\label{def_Ap}
	Let $w$ be a weight on $\R$ and let $1<p<\infty$. We say that
	\begin{enumerate}[label=\upshape(\alph*)]
		\item $w \in A_1$ if
		\[
		[w]_{A_1} \coloneq \sup_{I \subset \R} \left(\dashint_I w\right) \esssup_I (w^{-1}) < \infty;
		\]
		\item $w \in A_p$ if
		\[
		[w]_{A_p} \coloneq \sup_{I \subset \R} \left(\dashint_I w\right) \left(\dashint_I w^{-1/(p-1)}\right)^{p-1} < \infty;
		\]
		\item $w \in A_\infty$ if
		\[
		[w]_{A_\infty} \coloneq \sup_{I \subset \R} \left(\dashint_I w\right) \exp\left(\dashint_I \log(w^{-1})\right) < \infty;
		\]
		\item $w \in A_1^+$ if
		\[
		[w]_{A_1^+} \coloneq \sup_{I \subset \R} \left(\dashint_{I^-} w\right) \esssup_{I^+} (w^{-1}) < \infty;
		\]
		\item $w \in A_p^+$ if
		\[
		[w]_{A_p^+} \coloneq \sup_{I \subset \R} \left(\dashint_{I^-} w\right) \left(\dashint_{I^+} w^{-1/(p-1)}\right)^{p-1} < \infty;
		\]
		\item $w \in A_\infty^+$ if
		\[
		[w]_{A_\infty^+} \coloneq \sup_{I \subset \R} \left(\dashint_{I^-} w\right) \exp\left(\dashint_{I^+} \log(w^{-1})\right) < \infty.
		\]
	\end{enumerate}
\end{definition}

Another class of functions which play a central role in harmonic analysis are functions of bounded mean oscillation, which are defined as follows.

\begin{definition}
	\label{def_BMO}
	Let $f : \R \to \R$ be a locally integrable function. We say that
	\begin{enumerate}[label=\upshape(\alph*)]
		\item $f \in \BMO$ if
		\[
		\|f\|_{\BMO} \coloneq \sup_{I \subset \R} \dashint_I |f - f_I| < \infty;
		\]
		\item $f \in \BMO^+$ if
		\[
		\|f\|_{\BMO^+} \coloneq \sup_{I \subset \R} \dashint_{I^-} (f-f_{I^+})^+ < \infty.
		\]
	\end{enumerate}
\end{definition}

The classes $A_p^-$ (for $1 \le p \le \infty$) and $\BMO^-$ are defined similarly; simply interchange $I^-$ and $I^+$ in the definitions of $A_p^+$ (for $1 \le p \le \infty$) and $\BMO^+$.

It is well-known that a weight $w$ on $\R$ is an $A_p$ weight for some $1 \le p < \infty$ if and only if $w$ satisfies a reverse H\"older inequality, i.e.\ there exist $\epsilon,C>0$ such that, for all $I\subset\R$,
\begin{equation}
	\label{eq_RH1}
	\dashint_I w^{1+\epsilon} \le C \left(\dashint_I w\right)^{1+\epsilon}.
\end{equation}
The natural one-sided analogue of \eqref{eq_RH1} is the following inequality:
\begin{equation}
	\label{eq_RH2}
	\dashint_{I^-} w^{1+\epsilon} \le C \left(\dashint_{I^+} w\right)^{1+\epsilon}.
\end{equation}
Unfortunately, it is not true that $w \in A_p^+$ for some $1 \le p < \infty$ if and only if $w$ satisfies \eqref{eq_RH2} for some $\epsilon,C>0$ and all $I\subset\R$. However, the desired equivalence holds if we replace \eqref{eq_RH2} with the following weaker inequality, in which $I^+$ is replaced with $I$ on the right-hand side:
\begin{equation}
	\label{eq_RH3}
	\dashint_{I^-} w^{1+\epsilon} \le C \left(\dashint_I w\right)^{1+\epsilon}.
\end{equation}
See Proposition \ref{prop_ApBMO}(d) for a precise statement of this result.

We collect a number of useful facts about the $A_p^+$ and $\BMO^+$ classes in the following proposition. The reader will notice that many of these results are one-sided analogues of well-known properties of the usual $A_p$ and $\BMO$ classes.

\begin{proposition}
	\label{prop_ApBMO}
	The following hold:
	\begin{enumerate}[label=\upshape(\alph*)]
		\item $A_1^+ \subset A_p^+ \subset A_q^+ \subset A_\infty^+$ for all $1<p<q<\infty$.
		\item $A_q^+ = \bigcup_{1 \le p < q} A_p^+$ for all $1 < q \le \infty$.
		\item $A_p = A_p^+ \cap A_p^-$ for all $1 \le p \le \infty$.
		\item $w \in A_\infty^+$ if and only if there exist $\epsilon,C>0$ such that \eqref{eq_RH3} holds for all $I\subset\R$.
		\item If there exist $\epsilon,C>0$ such that \eqref{eq_RH2} holds for all $I\subset\R$, then $w \in A_\infty^+$.
		\item Let $1 \le p \le \infty$. If $w \in A_p^+$, then $w^\delta \in A_{1+\delta(p-1)}^+$ for all $0<\delta<1$.
		\item Let $1 \le p \le \infty$. If $w \in A_p^+$, then $w^{1+\delta} \in A_p^+$ for some $\delta>0$.
		\item For each $1<p<\infty$, we have $\BMO^+ = \{\alpha \log w : w \in A_p^+, \alpha\ge0\}$.
	\end{enumerate}
\end{proposition}

\begin{proof}
	(a) For any weight $w$, we have $[w]_{A_\infty^+} \le [w]_{A_q^+} \le [w]_{A_p^+} \le [w]_{A_1^+}$, where the first two inequalities follow from Jensen's inequality and the last inequality is trivial. The desired inclusions follow.
	
	(b) For $1<q<\infty$, see \cite[Remark C]{S1}. For $q = \infty$, see \cite[Theorem 1]{MPT}.
	
	(c) For $1 \le p <\infty$, see \cite[Theorem 4]{MOT}. The case $p = \infty$ follows from the case $1 \le p < \infty$, using (a), (b), and their analogues for usual Muckenhoupt weights.
	
	(d) See \cite[Theorem 6.5]{CNO}.
	
	(e) This follows from (d) since \eqref{eq_RH2} implies \eqref{eq_RH3} (with a different constant $C$).
	
	(f) Using Jensen's inequality, it is easy to show that $[w^\delta]_{A_q^+} \le [w]_{A_p^+}^\delta$, where $q=1+\delta(p-1)$. The desired implication follows.
	
	(g) For $p = 1$, see \cite[Remark C]{S1}. The case $1 < p < \infty$ follows from the case $p = 1$, using the factorization of $A_p^+$ weights stated in \cite[Remark B]{S1}. The case $p = \infty$ follows from the case $1 \le p < \infty$, using (a) and (b).
	
	(h) See \cite[Theorem 2]{MT}.
\end{proof}

The following proposition shows that the $A_p^+$ condition may be formulated in terms of a division of the interval $I$ into two subintervals which need not be of equal length.

\begin{proposition}
	\label{prop_Ap}
	Let $w$ be a weight on $\R$ and let $1<p<\infty$. Then $w \in A_p^+$ if and only if there exists $C>0$ such that, for all $a<b<c$, we have
	\begin{equation}
		\label{eq_Ap1}
		\left(\int_a^b w\right) \left(\int_b^c w^{-1/(p-1)}\right)^{p-1}
		\le C(c-a)^p.
	\end{equation}
\end{proposition}

\begin{proof}
	Suppose $w \in A_p^+$. For any $a<b<c$, there exists $I\subset\R$ such that $(a,b) \subset I^-$, $(b,c) \subset I^+$, and $|I^-| = |I^+| = \max(b-a,c-b)$. We estimate
	\begin{align*}
		(\tint_a^b w) (\tint_b^c w^{-1/(p-1)})^{p-1}
		&\le (\tint_{I^-} w) (\tint_{I^+} w^{-1/(p-1)})^{p-1} \\
		&\le [w]_{A_p^+} |I^-||I^+|^{p-1} \\
		&\le [w]_{A_p^+}(c-a)^p.
	\end{align*}
	Thus, if we take $C = [w]_{A_p^+}$, then \eqref{eq_Ap1} holds for all $a<b<c$.
	
	Conversely, suppose there exists $C>0$ such that \eqref{eq_Ap1} holds for all $a<b<c$. Given $I\subset\R$, write $I^- = (a,b)$ and $I^+ = (b,c)$. Then $|I^-|=|I^+|=\frac{1}{2}(c-a)$, so
	\begin{align*}
		(\tint_{I^-} w) (\tint_{I^+} w^{-1/(p-1)})^{p-1}
		&= (\tint_a^b w) (\tint_b^c w^{-1/(p-1)})^{p-1} \\
		&\le C(c-a)^p\\
		&= 2^p C |I^-||I^+|^{p-1}.
	\end{align*}
	Thus, $[w]_{A_p^+} \le 2^p C$, so $w \in A_p^+$.
\end{proof}

\subsection{Notions of porosity}

For any set $E \subset \R$ and any interval $I \subset \R$, we say that $I$ is \emph{$E$-free} if $I \cap E = \emptyset$.

\begin{definition}
	\label{def_length}
	Let $E \subset \R$ be a subset, let $0 < s \le 1$, and let $I \subset \R$ be an open interval. We define $L_s(E,I)$ to be the supremum of all $\ell>0$ for which there exist disjoint open intervals $I_1,\dots,I_k \subset I \setminus E$ (where $k \in \Z_{>0}$) such that $|I_i| \ge \ell$ for $i=1,\dots,k$ and $\sum_{i=1}^k |I_i| \ge (1-s)|I|$. If no such $\ell$ exists, we set $L_s(E,I) = 0$.
\end{definition}

When the set $E$ is clear from the context, we write $L_s(I)$ instead of $L_s(E,I)$. Note the following:
\begin{itemize}
	\item If $s<t$, then $L_s(E,I) \le L_t(E,I)$.
	\item If $E \subset F$, then $L_s(E,I) \ge L_s(F,I)$.
	\item $L_s(E,I) \le |I|$, with equality if and only if $I$ is $E$-free.
	\item If $I \setminus E$ contains an open interval, then $L_1(E,I)$ is the supremum of the lengths of the open intervals contained in $I \setminus E$. Otherwise, $L_1(E,I) = 0$.
\end{itemize}

The following two lemmas will be used frequently to simplify proofs.

\begin{lemma}
	\label{lem_length1}
	Let $E\subset\R$, $0<s\le1$, and $I\subset\R$.
	\begin{enumerate}[label=\upshape(\alph*)]
		\item $L_s(E,I) = L_s(\bar{E},I)$.
		\item If $E$ is closed and $|E|=0$, then $L_s(E,I) > 0$ and the supremum in the definition of $L_s(E,I)$ is achieved.
	\end{enumerate}
\end{lemma}

\begin{proof}
	(a) This follows from the fact that an open interval is disjoint from $E$ if and only if it is disjoint from $\bar{E}$.
	
	(b) The set $I \setminus E$ is open, so we may write $I \setminus E = \bigsqcup_{i \in \Gamma} (a_i,b_i)$, where $\Gamma = \{1,...,n\}$ for some $n\in\Z_{\ge0}$ or $\Gamma = \Z_{>0}$. Then $\sum_{i\in\Gamma} (b_i-a_i) = |I|$; in particular, this sum is finite, so we may assume $b_1 - a_1 \ge b_2 - a_2 \ge \cdots$. Any open interval contained in $I \setminus E$ is contained in $(a_i,b_i)$ for some $i \in \Gamma$, so we deduce the following: If $0<s<1$, then there is a unique $k\in\Gamma$ such that $\sum_{i=1}^{k-1} (b_i-a_i) < (1-s)|I|$ and $\sum_{i=1}^k (b_i-a_i) \ge (1-s)|I|$, and $L_s(E,I) = b_k-a_k$. If $s=1$, then $L_s(E,I) = b_1-a_1$. In either case, $L_s(E,I)>0$ and the supremum is achieved.
\end{proof}

\begin{lemma}
	\label{lem_length2}
	Let $E \subset \R$. Consider the following conditions:
	\begin{enumerate}[label=\upshape(\alph*)]
		\item $|\bar{E}|=0$.
		\item For all $0<s<1$ and all $I\subset\R$, we have $L_s(E,I)>0$.
		\item For some $0<s<1$ and all $I\subset\R$, we have $L_s(E,I)>0$.
		\item For all $I\subset\R$, we have $L_1(E,I)>0$.
	\end{enumerate}
	Then {\upshape(a)}, {\upshape(b)} and {\upshape(c)} are equivalent, and {\upshape(a)} implies {\upshape(d)}.
\end{lemma}

\begin{proof}
	By Lemma \ref{lem_length1}(a), we may assume $E$ is closed. If $|E|=0$, then, by Lemma \ref{lem_length1}(b), $L_s(E,I)>0$ for all $0 < s \le 1$ and $I \subset \R$. Thus, (a) implies both (b) and (d). Clearly, (b) implies (c). It remains to prove that (c) implies (a).
	
	Suppose (c) holds and $|E|>0$. Then, by Lebesgue's density theorem, there exists $x \in E$ such that $|E \cap I(x,r)|/|I(x,r)| \to 1$ as $r \downarrow 0$, where $I(x,r)=(x-r,x+r)$. For each $r>0$, we have $L_s(E,I(x,r))>0$, so there exist disjoint $I_1,\dots,I_k \subset I(x,r) \setminus E$ such that $\sum_{i=1}^k |I_i| \ge (1-s)|I(x,r)|$. This implies that $|E \cap I(x,r)| \le s|I(x,r)|$, which leads to a contradiction when we let $r \downarrow 0$. Thus, $|E|=0$, i.e.\ (a) holds. This completes the proof.
\end{proof}

\begin{remark}
	\label{rmk_length}
	In Lemma \ref{lem_length2}, (d) does not imply (a) in general. For example, let $E \subset [0, 1]$ be a Cantor set with $|E|>0$. Then $E$ is a compact, nowhere dense set. For every $I\subset\R$, we have $I \not\subset E$, so we may pick $x \in I \setminus E$. Then there exists $r>0$ such that $(x-r,x+r) \subset I \setminus E$, so $L_1(E,I)>0$.
\end{remark}

Of the four notions of porosity in the following definition, the first three were introduced in \cite{ALMV}, \cite{PU} and \cite{AGGM}, respectively, and the last one is introduced in the present paper.

\begin{definition}
	\label{def_porous}
	Let $E \subset \R$ be a subset, let $0<s<t<1$, and let $0<\delta<1$. We say that $E$ is
	\begin{enumerate}[label=\upshape(\alph*)]
		\item $(s,\delta)$-\emph{weakly porous} if, for every open interval $I\subset\R$, there exist disjoint open intervals $I_1,\dots,I_k \subset I \setminus E$ such that $|I_i| \ge \delta L_1(E,I)$ for $i=1,\dots,k$ and $\sum_{i=1}^k |I_i| \ge (1-s)|I|$;
		\item $(s,t,\delta)$-\emph{median porous} if, for every open interval $I\subset\R$, there exist disjoint open intervals $I_1,\dots,I_k \subset I \setminus E$ such that $|I_i| \ge \delta L_t(E,I)$ for $i=1,\dots,k$ and $\sum_{i=1}^k |I_i| \ge (1-s)|I|$;
		\item $(s,\delta)$-\emph{right weakly porous} if, for every open interval $I\subset\R$, there exist disjoint open intervals $I_1,\dots,I_k \subset I^-\setminus E$ such that $|I_i| \ge \delta L_1(E,I^+)$ for $i=1,\dots,k$ and $\sum_{i=1}^k |I_i| \ge (1-s)|I^-|$;
		\item $(s,t,\delta)$-\emph{right median porous} if, for every open interval $I\subset\R$, there exist disjoint open intervals $I_1,\dots,I_k \subset I^- \setminus E$ such that $|I_i| \ge \delta L_t(E,I^+)$ for $i=1,\dots,k$ and $\sum_{i=1}^k |I_i| \ge (1-s)|I^-|$.
	\end{enumerate}
\end{definition}

Left weakly porous and left median porous sets are defined similarly; simply interchange $I^-$ and $I^+$ in the definitions of right weakly porous and right median porous sets. We sometimes leave $\delta$, and maybe even $s$ and $t$, unspecified, to keep the notation manageable.

It is easy to see that, if $E$ is $(s,\delta)$-weakly porous (resp.\ $(s,\delta)$-right weakly porous), then $E$ is $(s,t,\delta)$-median porous (resp.\ $(s,t,\delta)$-right median porous) for all $t$ such that $s<t<1$.

The following lemma provides an alternative definition of right median porosity which is often useful.

\begin{lemma}
	\label{lem_porous1}
	Let $E\subset\R$, $0<s<t<1$, and $0<\delta<1$.
	\begin{enumerate}[label=\upshape(\alph*)]
		\item $E$ is $(s,t,\delta)$-right median porous if and only if $\bar{E}$ is $(s,t,\delta)$-right median porous.
		\item $E$ is $(s,t,\delta)$-right median porous if and only if $|\bar{E}|=0$ and $L_s(E,I^-) \ge \delta L_t(E,I^+)$ for all $I\subset\R$.
	\end{enumerate}
\end{lemma}

\begin{proof}
	(a) This follows from Lemma \ref{lem_length1}(a) and the fact stated in its proof.
	
	(b) By (a) and Lemma \ref{lem_length1}(a), we may assume $E$ is closed. Suppose $E$ is $(s,t,\delta)$-right median porous. Then $L_s(E,I^-) \ge \delta L_t(E,I^+)$ for all $I\subset\R$ and $L_s(E,I^-)>0$ for all $I\subset\R$. The latter is equivalent to saying that $L_s(E,I)>0$ for all $I\subset\R$, so $|E|=0$ by Lemma \ref{lem_length2}.
	
	Conversely, suppose $|E|=0$ and $L_s(E,I^-) \ge \delta L_t(E,I^+)$ for all $I\subset\R$. By Lemma \ref{lem_length1}(b), for all $I\subset\R$, we have $L_s(E,I^-)>0$ and the supremum in the definition of $L_s(E,I^-)$ is achieved. It follows that $E$ is $(s,t,\delta)$-right median porous.
\end{proof}

Evidently, Lemma \ref{lem_porous1} remains true (with the same proof) if we replace ``$(s,t,\delta)$-right median porous'' and ``$L_s(E,I^-) \ge \delta L_t(E,I^+)$'' with any one of the following:
\begin{itemize}
	\item ``$(s,\delta)$-right weakly porous'' and ``$L_s(E,I^-) \ge \delta L_1(E,I^+)$'';
	\item ``$(s,t,\delta)$-median porous'' and ``$L_s(E,I) \ge \delta L_t(E,I)$'';
	\item ``$(s,\delta)$-weakly porous'' and ``$L_s(E,I) \ge \delta L_1(E,I)$''.
\end{itemize}

\subsection{Medians}

\begin{definition}
	\label{def_median1}
	Let $I \subset \R$ be an open interval, let $f : I \to \R$ be a measurable function, and let $0<s<1$. We define the \emph{maximal median of $f$ over $I$ with parameter $s$} by
	\[
	M_s(f,I) = \sup\{\lambda\in\R : |\{x \in I : f(x) < \lambda\}| \le s|I|\}.
	\]
\end{definition}

When the function $f$ is clear from the context, we write $M_s(I)$ instead of $M_s(f,I)$. Note the following:
\begin{itemize}
	\item If $s<t$, then $M_s(f,I) \le M_t(f,I)$.
	\item If $f \le g$ a.e., then $M_s(f,I) \le M_s(g,I)$.
	\item For any $c\in\R$, we have $M_s(f+c,I) = M_s(f,I)+c$.
	\item For any $c>0$, we have $M_s(cf,I) = cM_s(f,I)$.
\end{itemize}

The following lemma contains several basic properties of medians. We include the proof for completeness.

\begin{lemma}
	\label{lem_median1}
	Let $f : I \to \R$ be a measurable function, let $0<s<t<1$, and let $c \in \R$.
	\begin{enumerate}[label=\upshape(\alph*)]
		\item $M_s(f,I)$ is a real number.
		\item $|\{x \in I : f(x) < M_s(f,I)\}| \le s|I|$.
		\item $|\{x \in I : f(x) > M_s(f,I)\}| \le (1-s)|I|$.
		\item If $f \ge c$ (resp.\ $f \le c$) a.e., then $M_s(f,I) \ge c$ (resp.\ $M_s(f,I) \le c$).
		\item If $f>c$ (resp.\ $f<c$) a.e., then $M_s(f,I) > c$ (resp.\ $M_s(f,I) < c$).
		\item If $I = I_1 \sqcup I_2$, then
		\[
		\min(M_s(f,I_1),M_s(f,I_2)) \le M_s(f,I) \le \max(M_s(f,I_1),M_s(f,I_2)).
		\]
		\item $M_{1-t}(-f,I) \le -M_s(f,I) \le M_{1-s}(-f,I)$.
	\end{enumerate}
	Let $J \subset \R$ be a (possibly unbounded) open interval such that $f(I) \subset J$ (up to a subset of $I$ of measure zero) and let $g : J \to \R$ be a continuous function.
	\begin{enumerate}[label=\upshape(\alph*), resume]
		\item If $g$ is strictly increasing, then $g(M_s(f,I)) = M_s(g \circ f,I)$.
		\item If $g$ is strictly decreasing, then $M_{1-t}(g\circ f, I) \le g(M_s(f,I)) \le M_{1-s}(g \circ f,I)$.
	\end{enumerate}
\end{lemma}

\begin{proof}
	(a) As $\lambda \downarrow -\infty$, we have $\{x \in I : f(x) < \lambda\} \downarrow \emptyset$ and hence $|\{x \in I : f(x) < \lambda\}| \downarrow 0$. As $\lambda \uparrow +\infty$, we have $\{x \in I : f(x) < \lambda\} \uparrow I$ and hence $|\{x \in I : f(x) < \lambda\}| \uparrow |I|$. Since the function $\lambda \mapsto |\{x \in I : f(x) < \lambda\}|$ is increasing, it follows that $M_s(f,I)$ is a real number.
	
	(b) As $\lambda \uparrow M_s(f,I)$, we have $\{x \in I : f(x) < \lambda\} \uparrow \{x \in I : f(x) < M_s(f,I)\}$ and hence $|\{x \in I : f(x) < \lambda\}| \uparrow |\{x \in I : f(x) < M_s(f,I)\}|$. Since $|\{x \in I : f(x) < \lambda\}| \le s|I|$ for all $\lambda < M_s(f,I)$, it follows that $|\{x \in I : f(x) < M_s(f,I)\}| \le s|I|$.
	
	(c) As $\lambda \downarrow M_s(f,I)$, we have $\{x \in I : f(x) < \lambda\} \downarrow \{x \in I : f(x) \le M_s(f,I)\}$ and hence $|\{x \in I : f(x) < \lambda\}| \downarrow |\{x \in I : f(x) \le M_s(f,I)\}|$. Since $|\{x \in I : f(x) < \lambda\}| > s|I|$ for all $\lambda > M_s(f,I)$, it follows that $|\{x \in I : f(x) \le M_s(f,I)\}| \ge s|I|$, so (c) holds.
	
	(d) If $f \ge c$ a.e.\ and $M_s(f,I) < c$, then, by (c), we have
	\[
	|I| = |\{x \in I : f(x) \ge c\}|
	\le |\{x \in I : f(x) > M_s(f,I)\}|
	\le (1-s)|I| < |I|,
	\]
	which is a contradiction. If $f \le c$ a.e.\ and $M_s(f,I) > c$, then, by (b), we have
	\[
	|I| = |\{x \in I : f(x) \le c\}|
	\le |\{x \in I : f(x) < M_s(f,I)\}|
	\le s|I| < |I|,
	\]
	which is a contradiction.
	
	(e) This can be proved in the same way as (d).
	
	(f) Let $\lambda_1 = M_s(f,I_1)$ and $\lambda_2 = M_s(f,I_2)$. If $\lambda < \min(\lambda_1,\lambda_2)$, then
	\[
	|\{x \in I : f(x) < \lambda\}|
	= |\{x \in I_1 : f(x) < \lambda\}| + |\{x \in I_2 : f(x) < \lambda\}|
	\le s|I_1| + s|I_2|
	= s|I|,
	\]
	so $\lambda \le M_s(f,I)$. Thus, $\min(\lambda_1,\lambda_2) \le M_s(f,I)$. If $\lambda>\max(\lambda_1,\lambda_2)$, then
	\[
	|\{x \in I : f(x) < \lambda\}|
	= |\{x \in I_1 : f(x) < \lambda\}| + |\{x \in I_2 : f(x) < \lambda\}|
	> s|I_1| + s|I_2|
	= s|I|,
	\]
	so $\lambda \ge M_s(f,I)$. Thus, $\max(\lambda_1,\lambda_2) \ge M_s(f,I)$.
	
	(g) By (c), we have
	\[
	|\{x \in I : -f(x) < -M_s(f,I)\}|
	= |\{x \in I : f(x) > M_s(f,I)\}|
	\le (1-s)|I|,
	\]
	so $M_{1-s}(-f,I) \ge -M_s(f,I)$. For every $\epsilon>0$, by (b), we have
	\begin{align*}
		|\{x \in I : -f(x) < -M_s(f,I) + \epsilon\}|
		&\ge |\{x \in I : -f(x) \le -M_s(f,I)\}| \\
		&= |\{x \in I : f(x) \ge M_s(f,I)\}| \ge (1-s)|I| > (1-t)|I|
	\end{align*}
	and hence $M_{1-t}(-f,I) \le -M_s(f,I) + \epsilon$, so $M_{1-t}(-f,I) \le -M_s(f,I)$.
	
	(h) By (e), we have $M_s(f,I) \in J$. For every sufficiently small $\epsilon>0$, we have
	\[
	|\{x \in I : (g \circ f)(x) < g(M_s(f,I) - \epsilon)\}|
	= |\{x \in I : f(x) < M_s(f,I) - \epsilon\}|
	\le s|I|
	\]
	and
	\[
	|\{x \in I : (g \circ f)(x) < g(M_s(f,I) + \epsilon)\}|
	= |\{x \in I : f(x) < M_s(f,I) + \epsilon\}|
	> s|I|,
	\]
	so $M_s(g \circ f,I) \ge g(M_s(f,I) - \epsilon)$ and $M_s(g \circ f,I) \le g(M_s(f,I) + \epsilon)$. Now let $\epsilon \to 0$.
	
	(i) Since $-g$ is strictly increasing, (h) gives $-g(M_s(f,I)) = M_s(-g \circ f,I)$. By (g), we have $M_{1-t}(g \circ f,I) \le -M_s(-g \circ f,I) \le M_{1-s}(g \circ f,I)$. The desired inequalities follow.
\end{proof}

The following three results on medians will be useful for our purposes.

\begin{lemma}
	\label{lem_median2}
	Let $f : I \to \R$ be a measurable function and let $0<s<1$. If $f>0$ a.e., then, for all $p \in \R$, we have $M_s(f,I)^p \lesssim_s \dashint_I f^p$.
\end{lemma}

\begin{proof}
	This was proved in \cite{PU}, but we include the proof for completeness. By Lemma \ref{lem_median1}(e), we have $M_s(f,I) > 0$. By Lemma \ref{lem_median1}(b) and Markov's inequality, we have
	\[
	(1-s)|I| \le |\{x \in I : f(x) \ge M_s(f,I)\}| \le \frac{1}{M_s(f,I)} \int_I f.
	\]
	Hence
	\begin{equation}
		\label{eq_median1}
		M_s(f,I) \le \frac{1}{1-s}~\dashint_I f.
	\end{equation}
	If $p=0$, there is nothing to prove. If $p>0$, then the function $t \mapsto t^p$ is strictly increasing on $(0,\infty)$, so, by Lemma \ref{lem_median1}(h) and \eqref{eq_median1}, we have $M_s(f,I)^p = M_s(f^p,I) \le \frac{1}{1-s}~\dashint_I f^p$. If $p<0$, then the function $t \mapsto t^p$ is strictly decreasing on $(0,\infty)$, so, by Lemma \ref{lem_median1}(i) and \eqref{eq_median1}, we have $M_s(f,I)^p \le M_{1-s}(f^p,I) \le \frac{1}{s}~\dashint_I f^p$.
\end{proof}

\begin{proposition}
	\label{prop_median}
	Let $f : I \to \R$ be a measurable function and let $0<s<1$. Then, for a.e.\ $x \in I$, and for every $x \in I$ at which $f$ is continuous, we have
	\[
	\lim_{I \ni x, I \downarrow x} M_s(f,I) = f(x).
	\]
\end{proposition}

\begin{proof}
	See \cite[Theorem 2.1]{PT}.
\end{proof}

\begin{lemma}
	\label{lem_median3}
	Let $f : [a,b] \to \R$ be a measurable function and let $0<s<1$. For $x \in (a,b]$, let $F(x) = M_s(f,[a,x])$.
	\begin{enumerate}[label=\upshape(\alph*)]
		\item $F$ is upper semicontinuous.
		\item If $f$ is continuous, then $F$ is continuous.
	\end{enumerate}
\end{lemma}

\begin{proof}
	(a) We need to show that, for any $C\in\R$ and any sequence $(x_n)\subset(a,b]$ converging to a point $x\in(a,b]$, if $F(x_n) \ge C$ for all $n$, then $F(x) \ge C$. Let $\lambda < C$. Then, for all $n$, we have $F(x_n) > \lambda$ and hence
	\[
	|\{y \in [a,x_n] : f(y) < \lambda\}| \le s|[a,x_n]|.
	\]
	Letting $n \to \infty$, we get
	\[
	|\{y \in [a,x] : f(y) < \lambda\}| \le s|[a,x]|.
	\]
	This implies that $F(x) \ge \lambda$. It follows that $F(x) \ge C$, as desired.
	
	(b) By (a), it suffices to show that $F$ is lower semicontinuous. In other words, we need to show that, for any $C\in\R$ and any sequence $(x_n)\subset(a,b]$ converging to a point $x\in(a,b]$, if $F(x_n) \le C$ for all $n$, then $F(x) \le C$. Let $\mu > \lambda > C$. Then, for all $n$, we have $F(x_n) < \lambda$ and hence
	\[
	|\{y \in [a,x_n] : f(y) < \lambda\}| > s|[a,x_n]|.
	\]
	Letting $n \to \infty$, we get
	\begin{equation}
		\label{eq_median2a}
		|\{y \in [a,x] : f(y) < \lambda\}| \ge s|[a,x]|.
	\end{equation}
	For the sake of contradiction, suppose
	\begin{equation}
		\label{eq_median2b}
		|\{y \in [a,x] : f(y) < \mu\}| \le s|[a,x]|.
	\end{equation}
	Then, since the left-hand side of \eqref{eq_median2a} is less than or equal to the left-hand side of \eqref{eq_median2b}, we have equality in both \eqref{eq_median2a} and \eqref{eq_median2b}. This implies that
	\begin{equation}
		\label{eq_median2c}
		|\{y \in [a,x] : \lambda \le f(y) < \mu\}| = 0.
	\end{equation}
	On the other hand, \eqref{eq_median2a} and \eqref{eq_median2b} imply that the sets $\{y \in [a,x] : f(y) < \lambda\}$ and $\{y \in [a,x] : f(y) \ge \mu\}$ are nonempty. Fix $\nu$ such that $\lambda < \nu < \mu$. By the intermediate value theorem, $f$ takes the value $\nu$ at a point of $[a,x]$. By continuity, there is a nonempty open subset $U$ of $[a,x]$ such that $f(U) \subset (\lambda,\mu)$. This contradicts \eqref{eq_median2c}, so \eqref{eq_median2b} is false. Hence $F(x) \le \mu$. It follows that $F(x) \le C$, as desired.
\end{proof}

\begin{remark}
	\label{rmk_median}
	The function $F$ in Lemma \ref{lem_median3} is not continuous in general. For example, let $[a,b] = [0,3]$, let $s=1/2$, and let $f(x) = 0$ for $0 \le x \le 1$ and $f(x) = 1$ for $1 < x \le 3$. Then $F(x) = 0$ for $0 \le x < 2$ and $F(x) = 1$ for $2 \le x \le 3$, so $F$ is not continuous.
\end{remark}

The following alternative notion of median will also be useful for our purposes.
\begin{definition}
	\label{def_median2}
	Let $I \subset \R$ be an open interval, let $f : I \to \R$ be a measurable function, and let $0<s<1$. We define the \emph{minimal median of $f$ over $I$ with parameter $s$} by
	\[
	m_s(f,I) = \inf\{\lambda\in\R : |\{x \in I : f(x) > \lambda\}| \le (1-s)|I|\}.
	\]
\end{definition}

The minimal median is related to the maximal median via the following equality, which follows easily from the definitions:
\begin{equation}
	\label{eq_median3a}
	m_s(f,I) = -M_{1-s}(-f,I).
\end{equation}
By Lemma \ref{lem_median1}(g) and \eqref{eq_median3a}, for any $0<s<t<1$, we have
\begin{equation}
	\label{eq_median3b}
	m_s(f,I) \le M_s(f,I) \le m_t(f,I).
\end{equation}

The minimal median satisfies properties analogous to those satisfied by the maximal median. For example, the following properties are evident:
\begin{itemize}
	\item If $s<t$, then $m_s(f,I) \le m_t(f,I)$.
	\item If $f \le g$ a.e., then $m_s(f,I) \le m_s(g,I)$.
	\item For any $c \in \R$, we have $m_s(f+c,I) = m_s(f,I)+c$.
	\item For any $c>0$, we have $m_s(cf,I) = cm_s(f,I)$.
\end{itemize}
Properties (a)-(f) in Lemma \ref{lem_median1} are also satisfied by the minimal median, and the following analogue of property (g) holds: If $s<t$, then
\[
m_{1-t}(-f,I) \le -m_t(f,I) \le m_{1-s}(-f,I).
\]
The analogue of property (f) will be used frequently, so we state it explicitly: If $I = I_1 \sqcup I_2$, then
\begin{equation}
	\label{eq_median3c}
	\min(m_s(f,I_1),m_s(f,I_2)) \le m_s(f,I) \le \max(m_s(f,I_1),m_s(f,I_2)).
\end{equation}
We shall also need the following result, which is an immediate consequence of Lemma \ref{lem_median3}(a) and \eqref{eq_median3a}.
\begin{lemma}
	\label{lem_median4}
	Let $f : [a,b] \to \R$ be a measurable function and let $0<s<1$. For $x \in (a,b]$, let $F(x) = m_s(f,[a,x])$. Then $F$ is lower semicontinuous.
\end{lemma}

\section{Characterizations of \texorpdfstring{$A_p^+$}{Ap+} and \texorpdfstring{$\BMO^+$}{BMO+}}
\label{sec_ApBMO}

In this section, we prove Theorems \ref{thm_Ap} and \ref{thm_BMO}, which characterize $A_p^+$ and $\BMO^+$, respectively, using medians. We begin with the proof of Theorem \ref{thm_Ap}.

\begin{proof}[Proof of Theorem \ref{thm_Ap}]
	It is clear that (b) implies (c), so it suffices to prove that (a) implies (b) and that (c) implies (a). As a preliminary observation, note that, since $w>0$ a.e., we have $M_s(w,I)>0$ for all $0<s<1$ and $I\subset\R$ by Lemma \ref{lem_median1}(e).
	
	First, we assume (a) and prove (b). There exists $K>0$ such that, for all $I\subset\R$, we have
	\begin{equation}
		\label{eq_Ap3a}
		\left(\dashint_{I^-} w\right) \left(\dashint_{I^+} w^{-1/(p-1)}\right)^{p-1} \le K.
	\end{equation}
	Let $0 < \gamma \le \frac{1}{2}$, $0<s<1$ and $I\subset\R$ be given. We have $I^{\gamma,\pm} \subset I^\pm$ and $|I^{\gamma,\pm}| = 2\gamma |I^\pm|$, so
	\begin{equation}
		\label{eq_Ap3b}
		\left(\dashint_{I^{\gamma,-}} w\right) \left(\dashint_{I^{\gamma,+}} w^{-1/(p-1)}\right)^{p-1}
		\le \frac{1}{(2\gamma)^p} \left(\dashint_{I^-} w\right) \left(\dashint_{I^+} w^{-1/(p-1)}\right)^{p-1}.
	\end{equation}
	By Lemma \ref{lem_median2}, we have
	\begin{equation}
		\label{eq_Ap3c}
		M_s(w,I^{\gamma,+})^{-1/(p-1)} \lesssim_s \dashint_{I^{\gamma,+}} w^{-1/(p-1)}.
	\end{equation}
	By putting \eqref{eq_Ap3a}, \eqref{eq_Ap3b} and \eqref{eq_Ap3c} together, we get \eqref{eq_Ap2} for some constant $C$ depending only on $p$, $K$, $\gamma$ and $s$, as required.
	
	Now, we assume (c) and prove (a). By \cite[Theorem 1]{MPT}, it suffices to show that there exist $0<\alpha<1$ and $\beta>0$ such that, for every $\lambda>0$ and $(a,b) \subset \R$, if
	\begin{equation}
		\label{eq_Ap3d}
		\dashint_{(a,b)} w = \lambda \le \dashint_{(a,x)} w
	\end{equation}
	for all $x \in (a,b)$, then
	\begin{equation}
		\label{eq_Ap3e}
		|\{x \in (a,b) : w(x) > \beta \lambda\}| > \alpha(b-a).
	\end{equation}
	(According to the statement of \cite[Theorem 1]{MPT}, a stronger version of this condition, in which ``there exists $0<\alpha<1$'' is replaced with ``for every $0<\alpha<1$'', implies that $w \in A_p^+$ for some $1<p<\infty$. However, from the proof of \cite[Theorem 1]{MPT}, we see that the weaker version stated above is sufficient.)
	
	Choose $\alpha$ and $\beta$ such that $0<\alpha<1-s$ and $0<\beta<1/C$. Suppose $\lambda>0$ and $(a,b)\subset\R$ satisfy \eqref{eq_Ap3d} for all $x \in (a,b)$. For each $k \in \Z_{\ge0}$, let $I_k$ be the open interval with the same left endpoint as $(a,b)$ and $(1-\gamma)^k$ times the length. Note that $(a,b) = \bigsqcup_{k=0}^\infty I_k^{\gamma,+}$. For each $k$, we have $I_k^{\gamma,-} = (a,x_k)$ for some $x_k \in (a,b)$, so, by \eqref{eq_Ap3d} and \eqref{eq_Ap2}, we have $\lambda \le C M_s(w,I_k^{\gamma,+})$. Hence $\beta \lambda < M_s(w,I_k^{\gamma,+})$, so, by Lemma \ref{lem_median1}(b), we have
	\[
	|\{x \in I_k^{\gamma,+} : w(x) \le \beta \lambda\}|
	\le |\{x \in I_k^{\gamma,+} : w(x) < M_s(w,I_k^{\gamma,+})\}|
	\le s|I_k^{\gamma,+}|.
	\]
	Summing over $k$, we get
	\[
	|\{x \in (a,b) : w(x) \le \beta\lambda\}| \le s(b-a).
	\]
	Therefore, \eqref{eq_Ap3e} holds, as required.
\end{proof}

In the rest of this section, we prove a number of results culminating with Theorem \ref{thm_BMO}. The first of these is the following key lemma, which will be used in the proof of Lemma \ref{lem_JN}.

\begin{lemma}
	\label{lem_BMO}
	Let $f : \R \to \R$ be a measurable function. Suppose there exist $0<s<t<1$ and $C>0$ such that, for all $I\subset\R$, we have
	\begin{equation}
		\label{eq_BMO1a}
		M_t(f,I^-) - M_s(f,I^+) \le C.
	\end{equation}
	Then, for all $I\subset\R$, $n\in\Z_{\ge0}$ and $J \in \D_n(I^l) \cup \D_n(I^c)$, we have
	\begin{equation}
		\label{eq_BMO1b}
		M_t(f,J) - M_s(f,I^r) \le 2(n+1)C.
	\end{equation}
\end{lemma}

\begin{proof}
	First, we show that, for all $I\subset\R$, we have
	\begin{align}
		\label{eq_BMO1c} M_s(I^-) - M_s(I) &\le C, \\
		\label{eq_BMO1d} M_t(I^-) - M_t(I) &\le C.
	\end{align}
	By Lemma \ref{lem_median1}(f), we have
	\[
	M_s(I) \ge \min(M_s(I^-), M_s(I^+)), \qquad
	M_t(I) \ge \min(M_t(I^-), M_t(I^+)).
	\]
	Hence
	\begin{align*}
		M_s(I^-) - M_s(I)
		&\le \max(0, M_s(I^-) - M_s(I^+))
		\le \max(0, M_t(I^-) - M_s(I^+))
		\le C, \\
		M_t(I^-) - M_t(I)
		&\le \max(0, M_t(I^-) - M_t(I^+))
		\le \max(0, M_t(I^-) - M_s(I^+))
		\le C.
	\end{align*}
	Thus, \eqref{eq_BMO1c} and \eqref{eq_BMO1d} hold.
	
	Now, we prove \eqref{eq_BMO1b} by induction. Note that
	\begin{align*}
		M_t(I^c) - M_s(I^r)
		&\le C, \\
		M_t(I^l) - M_s(I^r)
		&\le (M_t(I^l) - M_s(I^c)) + (M_t(I^c) - M_s(I^r))
		\le 2C.
	\end{align*}
	Thus, \eqref{eq_BMO1b} holds for $n=0$. Suppose \eqref{eq_BMO1b} holds for some $n \ge 0$. Let $K \in \D_{n+1}(I^l) \cup \D_{n+1}(I^c)$ be given. If $K = J^-$ for some $J \in \D_n(I^l) \cup \D_n(I^c)$, then
	\[
	M_t(K)
	= M_t(J^-)
	\le M_t(J) + C
	\le M_s(I^r) + (2(n+1)+1)C,
	\]
	where the first inequality follows from \eqref{eq_BMO1d} and the second inequality follows from the induction hypothesis. Otherwise, $K = J^+$ for some $J \in \D_n(I^l) \cup \D_n(I^c)$. First, suppose $J$ is not the rightmost interval in $\D_n(I^l) \cup \D_n(I^c)$. Then $J = H^-$ for some interval $H$ such that $H^+ \in \D_n(I^l) \cup \D_n(I^c)$, so
	\begin{align*}
		M_t(K)
		&= M_t((H^-)^+)
		\le M_s((H^+)^-) + C
		\le M_s(H^+) + 2C \\
		&\le M_t(H^+) + 2C
		\le M_s(I^r) + 2(n+2)C,
	\end{align*}
	where the first inequality follows from \eqref{eq_BMO1a}, the second inequality follows from \eqref{eq_BMO1c}, and the fourth inequality follows from the induction hypothesis. Now, suppose $J$ is the rightmost interval in $\D_n(I^l) \cup \D_n(I^c)$. Then $J=H^-$ for some interval $H$ such that $H^+$ is the leftmost interval in $\D_n(I^r)$, so
	\[
	M_t(K)
	= M_t((H^-)^+)
	\le M_s((H^+)^-) + C
	\le M_s(I^r) + (n+2)C,
	\]
	where the first inequality follows from \eqref{eq_BMO1a} and the second inequality follows from applying \eqref{eq_BMO1c} $n+1$ times.
	In all three cases, we have $M_t(K) - M_s(I^r) \le 2(n+2)C$, as required.
\end{proof}

The following technical lemma will also be used in the proof of Lemma \ref{lem_JN}.

\begin{lemma}
	\label{lem_median5}
	Let $f : \R \to \R$ be a measurable function. Let $0<s<1$ and $C \in \R$. For each $x \in \R$, we say that property $P(x)$ holds if there exists $y<x$ such that
	\[
	m_s(f,[y,x]) > C.
	\]
	Suppose $P(x)$ holds for all $x \in (a,b]$, but $P(a)$ does not hold. Then, for all $x \in (a,b]$, we have $M_s(f,[a,x]) \ge C$.
\end{lemma}

\begin{proof}
	If $m_s(f,[y,x]) > C$ for some $y<a$, then $m_s(f,[y,a]) \le C$ (since $P(a)$ does not hold), so $m_s(f,[a,x]) > C$ by \eqref{eq_median3c} and hence $M_s(f,[a,x]) > C$. Now, suppose
	\begin{equation}
		\label{eq_median4a}
		m_s(f,[y,x]) \le C \qquad \text{for all} ~ y<a.
	\end{equation}
	Let
	\[
	\beta = \inf \{y \in [a,x) : M_s(f,[y,x]) \ge C\}.
	\]
	Since $P(x)$ and \eqref{eq_median4a} hold, we have $m_s(f,[y_0,x]) > C$ for some $y_0 \in [a,x)$, so $\beta \le y_0$. Thus, $a \le \beta < x$. For each $n\in\Z_{>0}$, there exists $y_n \in [\beta, \beta+1/n)$ such that $y_n \in [a,x)$ and $M_s(f,[y_n,x]) \ge C$. By Lemma \ref{lem_median3}(a), this implies that
	\begin{equation}
		\label{eq_median4b}
		M_s(f,[\beta,x]) \ge C.
	\end{equation}
	If $\beta = a$, we are done. Assume $\beta > a$. Then $a < \beta < b$, so $P(\beta)$ holds, i.e.\ there exists $y < \beta$ such that
	\begin{equation}
		\label{eq_median4c}
		m_s(f,[y,\beta]) > C.
	\end{equation}
	By \eqref{eq_median4b}, \eqref{eq_median4c} and Lemma \ref{lem_median1}(f), we have $M_s(f,[y,x]) \ge C$, so, by the minimality of $\beta$, we must have $y < a$. Since $P(a)$ does not hold, we have
	\begin{equation}
		\label{eq_median4d}
		m_s(f,[y,a]) \le C.
	\end{equation}
	Then, by \eqref{eq_median4c}, \eqref{eq_median4d} and \eqref{eq_median3c}, we have
	\begin{equation}
		\label{eq_median4e}
		m_s(f,[a,\beta]) > C.
	\end{equation}
	Finally, by \eqref{eq_median4b}, \eqref{eq_median4e} and Lemma \ref{lem_median1}(f), we have $M_s(f,[a,x]) \ge C$, as desired.
\end{proof}

Next, we prove a lemma which contains the main construction that will be iterated in order to prove Theorem \ref{thm_JN}.

\begin{lemma}
	\label{lem_JN}
	Let $f : \R \to \R$ be measurable and let $0<s<\sigma<t<1$. Suppose $m,n \in \Z_{>0}$ satisfy $m(t-\sigma) > 1$ and $m2^{-n} < t-\sigma$. Suppose there exists $C>0$ such that, for all $I\subset\R$, we have
	\[
	M_t(f,I^-) - M_s(f,I^+) \le C.
	\]
	Then, for every $I\subset\R$, there exist countably many disjoint open intervals $I_i = (a_i,b_i) \subset I^l \cup I^c$ such that, for $g = (f-m_\sigma(f,I^r))^+ \chi_{I^l \cup I^c}$, we have
	\begin{enumerate}[label=\upshape(\alph*)]
		\item $m_\sigma(g,(x,b_i)) \le 4(n+1)C \le M_\sigma(g,(a_i,x))$ for all $i$ and $x \in I_i$;
		\item $m_\sigma(g,I_i) \le 4(n+1)C \le M_\sigma(g,I_i)$ for all $i$;
		\item $f(x)-m_\sigma(f,I^r) \le 4(n+1)C$ for a.e.\ $x \in I^l \setminus \bigsqcup_i I_i$;
		\item $\sum_i |I_i| \le \frac{1-t}{1-\sigma} (1+m2^{-n}) |I^l|$.
	\end{enumerate}
\end{lemma}

\begin{proof}
	Write $I^l = (a_I,b_I)$, $I^c = (b_I,c_I)$ and $I^r = (c_I,d_I)$. Note that $g : \R \to [0,\infty)$ is a measurable function which vanishes outside $(a_I,c_I)$. Let $\gamma = 4(n+1)C$ and
	\[
	A = \{x \in \R : \sup_{y<x} m_\sigma(g,(y,x)) > \gamma\}.
	\]
	In other words, let $A$ be the set of all $x \in \R$ for which there exists $y < x$ such that $m_\sigma(g,(y,x)) > \gamma$. We claim that $A$ has the following properties:
	\begin{enumerate}[label=\upshape(\Roman*)]
		\item $A \subset (a_I, \beta)$ for some $\beta \in (c_I, \infty)$.
		\item $A$ is an open set.
		\item If $x \in A \cap (c_I, \infty)$, then $[c_I, x] \subset A$.
	\end{enumerate}
	
	 Proof of (I): For all $x \le a_I$, we have $m_\sigma(g,(y,x)) = 0$ for all $y<x$ since $g$ is identically zero on $(y,x)$, so $x \notin A$. For all sufficiently large $x > c_I$, we have $m_\sigma(g,(y,x)) = 0$ for all $y<x$ since $g$ vanishes on more than $\sigma$ of $(y,x)$, so $x \notin A$. Thus, (I) holds.
	 
	 Proof of (II): Let $x \in A$. Then there exists $y<x$ such that $m_\sigma(g,(y,x)) > \gamma$. By Lemma \ref{lem_median4}, the function $z \mapsto m_\sigma(g,(y,z))$ is lower semicontinuous on $(y,\infty)$, so there exists $\epsilon>0$ such that $m_\sigma(g,(y,z)) > \gamma$ for all $z \in (x-\epsilon,x+\epsilon)$. Thus, $(x-\epsilon,x+\epsilon) \subset A$, so (II) holds.
	 
	 Proof of (III): Suppose $x \in A \cap (c_I,\infty)$. Then there exists $y<x$ such that $m_\sigma(g,(y,x)) > \gamma$. Since $m_\sigma(g,(z,x)) = 0$ for all $z \in [c_I,x)$, we must have $y < c_I$. It follows by \eqref{eq_median3c} that $m_\sigma(g,(y,z)) > \gamma$ for all $z \in [c_I,x)$. Thus, $[c_I,x] \subset A$, so (III) holds.
	
	From (I)-(III), we deduce that $A$ is the disjoint union of countably many open intervals, each of which is contained in $(a_I,c_I)$, with the possible exception of a single open interval $(a,b)$ with endpoints $a \in [a_I,c_I)$ and $b \in (c_I,\infty)$.
	
	Let $J = (a,b)$ be one of the connected components of $A$. Then $a \in [a_I,c_I)$ and $a,b \notin A$. Observe the following:
	
	\begin{enumerate}[label=\upshape(\roman*)]
		\item $M_\sigma(g,(a,x)) \ge \gamma$ for all $x \in J$. (This follows from Lemma \ref{lem_median5}.)
		\item $M_\sigma(g,(a,b)) \ge \gamma$. (This follows from (i) and Lemma \ref{lem_median3}(a).)
		\item $m_\sigma(g,(x,b)) \le \gamma$ for all $x \in J$.
		\item $m_\sigma(g,(a,b)) \le \gamma$. (This follows from (iii) and Lemma \ref{lem_median4}.)
	\end{enumerate}
	
	We claim that $J$ contains fewer than $m$ intervals of $\D_n(I^l) \cup \D_n(I^c)$. Suppose $J$ contains at least $m$ intervals of $\D_n(I^l) \cup \D_n(I^c)$. Let $K$ be the union of the $m$ leftmost intervals of $\D_n(I^l) \cup \D_n(I^c)$ that are contained in $J$. Let $\tilde{K}$ be the interval with the same left endpoint as $J$ and the same right endpoint as $K$. Then $K \subset \tilde{K}$ and $|\tilde{K} \setminus K| < 2^{-n}|I^l|$. Moreover, $\tilde{K} = (a,x)$ for some $x \in (a,b]$ and $\tilde{K} \subset (a_I,c_I)$. By Lemma \ref{lem_BMO}, for every $H \in \D_n(I^l) \cup \D_n(I^c)$, we have
	\begin{equation}
		\label{eq_JN1}
		M_t(f,H) - M_s(f,I^r) \le 2(n+1)C.
	\end{equation}
	By Lemma \ref{lem_median1}(f), this implies that
	\[
	M_t(f,K) - M_s(f,I^r) \le 2(n+1)C.
	\]
	By \eqref{eq_median3b}, it follows that
	\begin{equation}
		\label{eq_JN2}
		M_t(f,K) - m_\sigma(f,I^r) < \gamma.
	\end{equation}
	We estimate
	\begin{align*}
		(1-\sigma)|K|
		&\le (1-\sigma)|\tilde{K}| \\
		&\le |\{x \in \tilde{K} : g(x) \ge M_\sigma(g,\tilde{K})\}| \\
		&\le |\{x \in \tilde{K} : g(x) \ge \gamma\}| \\
		&= |\{x \in \tilde{K} : f(x)-m_\sigma(f,I^r) \ge \gamma\}| \\
		&\le |\{x \in \tilde{K} : f(x) > M_t(f,K)\}| \\
		&\le |\tilde{K} \setminus K| + |\{x \in K : f(x) > M_t(f,K)\}| \\
		&\le 2^{-n}|I^l| + (1-t)|K|.
	\end{align*}
	Here we used Lemma \ref{lem_median1}(b), observation (i) or (ii), the definition of $g$, inequality \eqref{eq_JN2}, and Lemma \ref{lem_median1}(c). Hence $(t-\sigma)|K| \le 2^{-n}|I^l|$. Since $|K| = m2^{-n}|I^l|$, this contradicts our assumption that $m(t-\sigma) > 1$. Thus, our claim has been proved.
	
	Now, let $\{I_i\}$ be the collection of connected components of $A$ which intersect $I^l$. Let $K$ be the union of the $m$ leftmost intervals in $\D_n(I^c)$. (This definition makes sense since $m < 2^n$ by assumption.) For each $i$, we have $I_i \subset I^l \cup K$ by the preceding paragraph. In particular, $I_i \subset I^l \cup I^c$.
	
	It remains to check that the collection $\{I_i\}$ satisfies properties (a)-(d) in the statement of the lemma. Properties (a) and (b) follow immediately from observations (i)-(iv). To prove (c), note that, for any $x \in \R \setminus A$, we have $m_\sigma(g,(y,x)) \le \gamma$ for all $y<x$. By \eqref{eq_median3b}, this implies that $M_s(g,(y,x)) \le \gamma$ for all $y<x$. By Proposition \ref{prop_median}, we have $M_s(g,(y,x)) \to g(x)$ as $y \uparrow x$ for a.e.\ $x \in \R$, so $g(x) \le \gamma$ for a.e.\ $x \in \R \setminus A$. Recalling the definition of $g$ and of the collection $\{I_i\}$, we see that (c) holds. Finally, to prove (d), note that, by \eqref{eq_JN1} and Lemma \ref{lem_median1}(f), we have
	\[
	M_t(f,I^l \cup K) - M_s(f,I^r) \le 2(n+1)C.
	\]
	By \eqref{eq_median3b}, this implies that
	\begin{equation}
		\label{eq_JN3}
		M_t(f,I^l \cup K) - m_\sigma(f,I^r) < \gamma.
	\end{equation}
	For each $i$, we estimate
	\begin{align*}
		(1-\sigma)|I_i|
		&\le |\{x \in I_i : g(x) \ge M_\sigma(g,I_i)\}| \\
		&\le |\{x \in I_i : g(x) \ge \gamma\}| \\
		&= |\{x \in I_i : f(x) - m_\sigma(f,I^r) \ge \gamma\}| \\
		&\le |\{x \in I_i : f(x) > M_t(f,I^l \cup K)\}|.
	\end{align*}
	Here we used Lemma \ref{lem_median1}(b), observation (ii), the definition of $g$, and inequality \eqref{eq_JN3}. By Lemma \ref{lem_median1}(c), it follows that
	\begin{align*}
		(1-\sigma) \sum_i |I_i|
		&\le |\{x \in I^l \cup K : f(x) > M_t(f,I^l \cup K)\}| \\
		&\le (1-t) |I^l \cup K| \\
		&= (1-t)(1+m2^{-n})|I^l|.
	\end{align*}
	Thus, (d) holds.
\end{proof}

Now, we are ready to prove a John-Nirenberg inequality for functions whose one-sided median differences are bounded.

\begin{theorem}
	\label{thm_JN}
	Let $f : \R \to \R$ be measurable and let $0<s<\sigma<t<1$. Suppose there exists $C>0$ such that, for all $I\subset\R$, we have
	\[
	M_t(f,I^-) - M_s(f,I^+) \le C.
	\]
	Then there exist $K,\alpha>0$ (depending only on $\sigma$ and $t$) such that, for all $I\subset\R$ and $\lambda>0$, we have
	\[
	|\{x \in I^l : f(x) - M_\sigma(f,I^r) > \lambda\}|
	\le K |I^l| \exp(-\alpha \lambda /C).
	\]
\end{theorem}

\begin{proof}
	Choose $M \in \Z_{>0}$ large enough that $M(t-\sigma) > 1$, and choose $N \in \Z_{>0}$ large enough that $M2^{-N} < t-\sigma$. Let $\beta = \frac{1-t}{1-\sigma}(1+M2^{-N})$ and $\gamma = 4(N+1)C$. Note that
	\[
	\beta < \frac{1-t}{1-\sigma}(1+t-\sigma) = 1-t\left(1-\frac{1-t}{1-\sigma}\right) < 1.
	\]
	Thus, $0<\beta<1$ and $\gamma>0$.
	
	Let $I\subset\R$ be given. By Lemma \ref{lem_JN}, there exist countably many disjoint open intervals $I_i = (a_i,b_i) \subset I^l \cup I^c$ such that, for $g = (f-m_\sigma(f,I^r))^+ \chi_{I^l \cup I^c}$, we have
	\begin{enumerate}[label=\upshape(\alph*)]
		\item $m_\sigma(g,(x,b_i)) \le \gamma \le M_\sigma(g,(a_i,x))$ for all $i$ and $x \in I_i$;
		\item $m_\sigma(g,I_i) \le \gamma \le M_\sigma(g,I_i)$ for all $i$;
		\item $f(x)-m_\sigma(f,I^r) \le \gamma$ for a.e.\ $x \in I^l \setminus \bigsqcup_i I_i$;
		\item $\sum_i |I_i| \le \beta |I^l|$.
	\end{enumerate}
	For each $i$ and each $j \in \Z_{\ge0}$, let $I_{ij}$ be the open interval with the same right endpoint as $I_i$ and $(2/3)^j$ times the length. Note that $I_i = \bigsqcup_{j=0}^\infty I_{ij}^l$. By Lemma \ref{lem_JN}, for each $i$ and $j$, there exist countably many disjoint open intervals $I_{ijk} = (a_{ijk},b_{ijk}) \subset I_{ij}^l \cup I_{ij}^c$ such that, for $g_{ij} = (f-m_\sigma(f,I_{ij}^r))^+ \chi_{I_{ij}^l \cup I_{ij}^c}$, we have
	\begin{enumerate}[label=\upshape(\alph*')]
		\item $m_\sigma(g_{ij},(x,b_{ijk})) \le \gamma \le M_\sigma(g_{ij},(a_{ijk},x))$ for all $k$ and $x \in I_{ijk}$;
		\item $m_\sigma(g_{ij},I_{ijk}) \le \gamma \le M_\sigma(g_{ij},I_{ijk})$ for all $k$;
		\item $f(x)-m_\sigma(f,I_{ij}^r) \le \gamma$ for a.e.\ $x \in I_{ij}^l \setminus \bigsqcup_k I_{ijk}$;
		\item $\sum_k |I_{ijk}| \le \beta |I_{ij}^l|$.
	\end{enumerate}
	
	Now, $\{I_{ijk}\}_{i,j,k}$ is a countable collection of (not necessarily disjoint) open subintervals of $I$. For any $i$ and $j$, we have
	\begin{align*}
		m_\sigma(f,I_{ij}^r) - m_\sigma(f,I^r)
		&= m_\sigma(f-m_\sigma(f,I^r), I_{ij}^r) \\
		&\le m_\sigma((f-m_\sigma(f,I^r))^+, I_{ij}^r) \\
		&= m_\sigma(g,I_{ij}^r) \\
		&\le \gamma,
	\end{align*}
	where the second equality holds since $I_{ij}^r \subset I^l \cup I^c$, and the second inequality holds since $I_{ij}^r = (x,b_i)$ for some $x \in I_i$. Together with (c'), this implies that
	\begin{equation}
		\label{eq_JN4}
		f(x) - m_\sigma(f,I^r) \le 2\gamma
	\end{equation}
	for a.e.\ $x \in I_{ij}^l \setminus \bigsqcup_k I_{ijk}$. Since $I_i = \bigsqcup_j I_{ij}^l$, it follows that \eqref{eq_JN4} holds for a.e.\ $x \in \bigsqcup_i I_i \setminus \bigcup_{i,j,k} I_{ijk}$. Together with (c), this implies that \eqref{eq_JN4} holds for a.e.\ $x \in I^l \setminus \bigcup_{i,j,k} I_{ijk}$. Also note that
	\[
	\Bigl|\bigcup_{i,j,k} I_{ijk}\Bigr|
	\le \sum_{i,j,k} |I_{ijk}|
	\le \beta \sum_{i,j} |I_{ij}^l|
	= \beta \sum_i |I_i|
	\le \beta^2 |I^l|,
	\]
	where the second inequality follows from (d'), the equality holds since $I_i = \bigsqcup_j I_{ij}^l$, and the third inequality follows from (d).
	
	By continuing this process indefinitely, we obtain for each $n \in \Z_{>0}$ a countable collection $\{J_{ni}\}_i$ of (not necessarily disjoint) open subintervals of $I$ such that
	\begin{enumerate}[label=\upshape(\roman*)]
		\item $f(x) - m_\sigma(f,I^r) \le n\gamma$ for a.e.\ $x \in I^l \setminus \bigcup_i J_{ni}$;
		\item $|\bigcup_i J_{ni}| \le \beta^n |I^l|$.
	\end{enumerate}
	By \eqref{eq_median3b}, we may replace $m_\sigma(f,I^r)$ with $M_\sigma(f,I^r)$ in (i).
	
	Let $\lambda>0$ be given. Then there is a unique $n \in \Z_{\ge0}$ such that
	\[
	n\gamma < \lambda \le (n+1)\gamma.
	\]
	If $n \ge 1$, then
	\begin{align*}
		|\{x \in I^l : f(x) - M_\sigma(f,I^r) > \lambda\}|
		&\le |\{x \in I^l : f(x) - M_\sigma(f,I^r) > n\gamma\}| \\
		&\le \Bigl|\bigcup_i J_{ni}\Bigr|
		\le \beta^n |I^l|
		\le |I^l| \exp\left(\frac{\log\beta}{8(N+1)}\cdot\frac{\lambda}{C}\right),
	\end{align*}
	where the last inequality holds since $\lambda \le 2n\gamma = 8n(N+1)C$. If $n = 0$, then
	\[
	|\{x \in I^l : f(x) - M_\sigma(f,I^r) > \lambda\}|
	\le |I^l|
	\le \frac{1}{\beta} |I^l| \exp\left(\frac{\log\beta}{8(N+1)}\cdot\frac{\lambda}{C}\right),
	\]
	where the last inequality holds since $\lambda \le 2\gamma = 8(N+1)C$. Therefore, we may take $K = 1/\beta$ and $\alpha = -(\log\beta)/(8(N+1))$.
\end{proof}

Finally, we are ready to prove Theorem \ref{thm_BMO}.

\begin{proof}[Proof of Theorem \ref{thm_BMO}]
	It is clear that (b) implies (c), so it suffices to prove that (a) implies (b) and that (c) implies (a).
	
	First, we assume (a) and prove (b). Let $0<s<t<1$ be given. Fix $1<p<\infty$. By Proposition \ref{prop_ApBMO}(h), there exists $\beta>0$ such that $\exp(\beta f) \in A_p^+$. By Theorem \ref{thm_Ap}, there exists $K>0$ such that, for all $I\subset\R$,
	\[
	\dashint_{I^-} \exp(\beta f) \le K M_s(\exp(\beta f),I^+).
	\]
	Given $I\subset\R$, we estimate
	\begin{align*}
		\exp(\beta(M_t(f,I^-) - M_s(f,I^+)))
		&= \frac{\exp(\beta M_t(f,I^-))}{\exp(\beta M_s(f,I^+))}
		= \frac{M_t(\exp(\beta f),I^-)}{M_s(\exp(\beta f),I^+)} \\
		&\lesssim_t \frac{\dashint_{I^-} \exp(\beta f)}{M_s(\exp(\beta f),I^+)}
		\le K,
	\end{align*}
	where the second equality follows from Lemma \ref{lem_median1}(h) and the first inequality follows from Lemma \ref{lem_median2}. It follows that \eqref{eq_BMO2} holds with a constant $C$ depending only on $t$, $\beta$ and $K$, as required.
	
	Now, we assume (c) and prove (a). Fix $\sigma$ such that $s<\sigma<t$. Let $K$ and $\alpha$ be as given by Theorem \ref{thm_JN}. Fix $\beta$ such that $0 < \beta < \alpha/C$. Given $I\subset\R$, we estimate
	\begin{align*}
		\int_{I^l} \exp(\beta (f-M_\sigma(f,I^r)))
		&= \int_0^\infty |\{x \in I^l : \exp(\beta(f(x)-M_\sigma(f,I^r))) > \lambda\}| \, d\lambda \\
		&\le |I^l| + \int_1^\infty |\{x \in I^l : f(x) - M_\sigma(f,I^r) > (\log \lambda)/\beta\}| \, d\lambda \\
		&\le |I^l| + K|I^l| \underbrace{\int_1^\infty \lambda^{-\alpha/(\beta C)} \, d\lambda}_{\eqcolon\,\gamma}.
	\end{align*}
	Note that $\gamma < \infty$ by our choice of $\beta$. Let $L = 1 + K\gamma$. Then
	\[
	\left(\dashint_{I^l} \exp(\beta f)\right) \exp(-\beta M_\sigma(f,I^r))
	= \dashint_{I^l} \exp(\beta(f-M_\sigma(f,I^r)))
	\le L.
	\]
	By Lemma \ref{lem_median1}(h), we have
	\[
	\exp(-\beta M_\sigma(f,I^r))
	= \exp(-M_\sigma(\beta f, I^r))
	= \exp(M_\sigma(\beta f, I^r))^{-1}
	= M_\sigma(\exp(\beta f), I^r)^{-1}.
	\]
	Hence
	\[
	\dashint_{I^l} \exp(\beta f) \le L \cdot M_\sigma(\exp(\beta f), I^r).
	\]
	In particular, $\exp(\beta f)$ is a weight. Note that $L$ and $\beta$ depend only on $\sigma$, $t$ and $C$. Thus, by Theorem \ref{thm_Ap}, we have $\exp(\beta f) \in A_p^+$ for some $1<p<\infty$. It follows by Proposition \ref{prop_ApBMO}(h) that $f \in \BMO^+$, as desired.
\end{proof}

\section{\texorpdfstring{$A_p^+$}{Ap+} is equivalent to right median porosity}
\label{sec_main}

In this section, we prove Theorem \ref{thm_main}, which characterizes the sets $E$ such that the weight $d_E^{\,-\alpha}$ belongs to $A_p^+$ for some $\alpha>0$ and $1<p<\infty$. We begin with the following lemma, which is partly based on some of the ideas in \cite{PU}.

\begin{lemma}
	\label{lem_dist1}
	Let $\emptyset \ne E \subset \R$, $0<s<t<1$, $c>0$, and $I\subset\R$.
	\begin{enumerate}[label=\upshape(\alph*)]
		\item If $I \cap E \ne \emptyset$, then $M_s(d_E,I) \le L_s(E,I) \lesssim_{s,t} M_t(d_E,I)$.
		\item If $I \cap E = \emptyset$ and $d(I,E) \le c|I|$, then $L_s(E,I) \approx_{s,c} M_s(d_E,I)$.
		\item If $|\bar{E}|=0$ and $d(I,E) \le c|I|$, then, for any $p\in\R$, we have $L_s(E,I)^p \lesssim_{s,p,c} \dashint_I d_E^{\,p}$.
		\item If $I^- \cap E = I^+ \cap E = \emptyset$, then $M_s(d_E,I^-) \approx_{s,t} M_t(d_E,I^+)$.
	\end{enumerate}
\end{lemma}

\begin{proof}
	By Lemma \ref{lem_length1}(a) and the fact that $d_E = d_{\bar{E}}$, we may assume $E$ is closed. For any $\delta>0$, let $S_\delta(E,I)$ be the collection of connected components $J$ of $I \setminus E$ such that $|J| \ge \delta$.
	
	(a) First, we prove the left-hand inequality. Since $I \cap E \ne \emptyset$, we have $L_s(E,I) < |I|$. Let $L_s(E,I) < \delta < |I|$. If $x \in I$ and $d_E(x) \ge \delta$, then $x \in J$ for some $J \in S_\delta(E,I)$. (Indeed, $x \notin E$, so $x \in J$ for some component $J$ of $I \setminus E$. If $|J| < \delta$, then $|J|<|I|$, so at least one endpoint of $J$ belongs to $E$ and hence $d_E(x) < |J| < \delta$, which is a contradiction.) Thus,
	\[
	|\{x \in I : d_E(x) \ge \delta\}|
	\le \sum_{J \in S_\delta(E,I)} |J|
	< (1-s)|I|.
	\]
	Hence $|\{x \in I : d_E(x) < \delta\}| > s|I|$, so $M_s(d_E,I) \le \delta$. Thus, $M_s(d_E,I) \le L_s(E,I)$.
	
	Now, we prove the right-hand inequality. Let $c = \frac{1-t}{1-s}$. Assume $L_s(E,I) > 0$ and let $0 < \delta < L_s(E,I)$. If $J \in S_\delta(E,I)$ and $x \in cJ$, then $d_E(x) \ge \frac{1-c}{2}|J| \ge \frac{1-c}{2}\delta$. Thus,
	\[
	|\{x \in I : d_E(x) \ge \tfrac{1-c}{2}\delta\}|
	\ge \sum_{J \in S_\delta(E,I)} |cJ|
	= c \sum_{J \in S_\delta(E,I)} |J|
	\ge c(1-s)|I|
	= (1-t)|I|.
	\]
	Hence $|\{x \in I : d_E(x) < \frac{1-c}{2}\delta\}| \le t|I|$, so $M_t(d_E,I) \ge \frac{1-c}{2}\delta$. Thus, $M_t(d_E,I) \ge \frac{1-c}{2}L_s(E,I)$.
	
	(b) Since $I \cap E = \emptyset$, we have $L_s(E,I) = |I|$. For any $x \in I$, we have $d_E(x) \le |I| + d(I,E) \le (c+1)|I|$. By Lemma \ref{lem_median1}(d), $M_s(d_E,I) \le (c+1)|I|$. On the other hand, if $x \in (1-s)I$, then $d_E(x) \ge \frac{s}{2}|I|$. Thus, $|\{x \in I : d_E(x) < \frac{s}{2}|I|\}| \le s|I|$, so $M_s(d_E,I) \ge \frac{s}{2} |I|$.
	
	(c) Since $|\bar{E}| = 0$, we have $d_E > 0$ a.e. By Lemma \ref{lem_median2}, it suffices to show that $L_s(E,I)^p \lesssim_{s,p,c} M_\sigma(d_E,I)^p$ for some $\sigma=\sigma(s) \in (0,1)$. If $I \cap E = \emptyset$, this follows immediately from (b), so we assume $I \cap E \ne \emptyset$. Taking $t = \frac{s+1}{2}$ in (a), we get
	\[
	M_s(d_E, I) \le L_s(E,I) \lesssim_s M_{(s+1)/2}(d_E,I).
	\]
	If $p\ge0$, the right-hand inequality gives $L_s(E,I)^p \lesssim_{s,p} M_{(s+1)/2}(d_E,I)^p$. If $p\le0$, the left-hand inequality gives $L_s(E,I)^p \le M_s(d_E,I)^p$. In either case, the desired inequality holds.
	
	(d) For any $x \in I^+$, we have $d_E(x) \le |I^+| + d(I^+,E)$. Since $|d(I^-,E) - d(I^+,E)| \le \frac{1}{2}|I|$, this implies that $d_E(x) \le 2|I^-| + d(I^-,E)$. By Lemma \ref{lem_median1}(d), $M_t(d_E,I^+) \le 2|I^-| + d(I^-,E)$. On the other hand, if $x \in (1-s)I^-$, then $d_E(x) \ge \frac{s}{2}|I^-| + d(I^-,E)$. Thus,
	\[
	|\{x \in I^- : d_E(x) < \tfrac{s}{2}|I^-| + d(I^-,E)\}| \le s|I^-|,
	\]
	so $M_s(d_E,I^-) \ge \frac{s}{2}|I^-| + d(I^-,E)$. It follows that
	\[
	\frac{M_t(d_E,I^+)}{M_s(d_E,I^-)}
	\le \frac{2|I^-| + d(I^-,E)}{\frac{s}{2}|I^-| + d(I^-,E)}
	\le \frac{4}{s}.
	\]
	Similarly, $M_s(d_E,I^-)/M_t(d_E,I^+) \le 4/t$.
\end{proof}

The following two lemmas contain the two key implications in Theorem \ref{thm_main}.

\begin{lemma}
	\label{lem_main1}
	Let $\emptyset \ne E \subset \R$. Suppose there exist $\alpha>0$ and $1<p<\infty$ such that $d_E^{\,-\alpha} \in A_p^+$. Then $E$ is $(s,t)$-right median porous for all $0<s<t<1$.
\end{lemma}

\begin{proof}
	By Lemma \ref{lem_porous1}(a) and the fact that $d_E = d_{\bar{E}}$, we may assume $E$ is closed. Since $d_E^{\,-\alpha}$ is a weight, we have $d_E^{\,-\alpha} < \infty$ a.e., so $d_E > 0$ a.e.\ and hence $|E|=0$. Let $0<s<t<1$ be given. By Lemma \ref{lem_porous1}(b), it suffices to show that there exists $0<\delta<1$ such that $L_s(I^-) \ge \delta L_t(I^+)$ for all $I\subset\R$.
	
	Let $0<\delta<1$ be a constant whose precise value will be determined later. Given $I\subset\R$, we may write $I^- \setminus E = \bigsqcup_{i \in \Gamma} (a_i,b_i)$ and $I^+ \setminus E = \bigsqcup_{j \in \Delta} (c_j,d_j)$, where $\Gamma$ and $\Delta$ are countable sets. Then $\sum_{i \in \Gamma} (b_i-a_i) = |I^-|$ and $\sum_{j \in \Delta} (d_j-c_j) = |I^+|$. Let
	\[
	\Gamma_0 = \{i \in \Gamma : b_i - a_i \ge \delta L_t(I^+)\}
	\qquad \text{and} \qquad
	\Delta_0 = \{j \in \Delta : d_j - c_j \ge L_t(I^+)\}.
	\]
	By Lemma \ref{lem_length1}(b), we have $L_t(I^+) > 0$, so $\Gamma_0$ and $\Delta_0$ are finite sets. It suffices to show that $\sum_{i \in \Gamma_0} (b_i-a_i) \ge (1-s)|I^-|$ or, equivalently, that $\sum_{i \in \Gamma \setminus \Gamma_0} (b_i-a_i) \le s|I^-|$.
	
	Let $i \in \Gamma \setminus \Gamma_0$ and $x \in (a_i,b_i)$. Then $b_i-a_i < \delta L_t(I^+) \le \delta |I^+| < |I^+| = |I^-|$, so either $a_i \in E$ or $b_i \in E$. Hence
	\[
	0 < \min(x-a_i,b_i-x) \le d_E(x) \le \max(x-a_i,b_i-x) < b_i-a_i < \delta L_t(I^+).
	\]
	This implies that $1 \le \delta^\alpha L_t(I^+)^\alpha d_E(x)^{-\alpha}$, so
	\begin{equation}
		\label{eq_main1}
		\frac{1}{|I^-|} \sum_{i \in \Gamma \setminus \Gamma_0} (b_i-a_i)
		\le \frac{\delta^\alpha L_t(I^+)^\alpha}{|I^-|} \sum_{i \in \Gamma \setminus \Gamma_0} \int_{a_i}^{b_i} d_E^{\,-\alpha}
		\le \delta^\alpha L_t(I^+)^\alpha \dashint_{I^-} d_E^{\,-\alpha}.
	\end{equation}
	On the other hand, by Lemma \ref{lem_length1}(b), we have
	\begin{equation}
		\label{eq_main2}
		\sum_{j \in \Delta_0} (d_j-c_j) \ge (1-t)|I^+|.
	\end{equation}
	For any $j \in \Delta_0$, since $d_E(x) \ge \min(x-c_j,d_j-x)$ for all $x \in (c_j,d_j)$, we have
	\[
	\int_{c_j}^{d_j} d_E^{\,\alpha/(p-1)}
	\ge 2 \int_0^{(d_j-c_j)/2} x^{\alpha/(p-1)} \,dx
	= \frac{2}{\frac{\alpha}{p-1}+1} \left(\frac{d_j-c_j}{2}\right)^{\frac{\alpha}{p-1}+1}.
	\]
	Hence
	\begin{equation}
		\label{eq_main3}
		\sum_{j \in \Delta_0} (d_j-c_j)^{\frac{\alpha}{p-1}+1}
		\lesssim_{\alpha,p} \sum_{j \in \Delta_0} \int_{c_j}^{d_j} d_E^{\,\alpha/(p-1)}
		\le \int_{I^+} d_E^{\,\alpha/(p-1)}.
	\end{equation}
	By H\"older's inequality,
	\begin{equation}
		\label{eq_main4}
		\Bigl(\sum_{j \in \Delta_0} (d_j-c_j)\Bigr)^{\alpha+p-1}
		\le \Bigl(\sum_{j \in \Delta_0} (d_j-c_j)^{\frac{\alpha}{p-1}+1}\Bigr)^{p-1} \mathrm{card}(\Delta_0)^\alpha.
	\end{equation}
	Note that
	\begin{equation}
		\label{eq_main5}
		\mathrm{card}(\Delta_0) L_t(I^+) \le \sum_{j \in \Delta_0} (d_j-c_j) \le |I^+|.
	\end{equation}
	Putting \eqref{eq_main2}, \eqref{eq_main3}, \eqref{eq_main4} and \eqref{eq_main5} together, we get
	\begin{equation}
		\label{eq_main6}
		1 \lesssim_{\alpha,p} \frac{L_t(I^+)^{-\alpha}}{(1-t)^{\alpha+p-1}} \left(\dashint_{I^+} d_E^{\,\alpha/(p-1)}\right)^{p-1}.
	\end{equation}
	Putting \eqref{eq_main1} and \eqref{eq_main6} together and recalling that $d_E^{\,-\alpha} \in A_p^+$, we get
	\[
	\frac{1}{|I^-|} \sum_{i \in \Gamma \setminus \Gamma_0} (b_i-a_i)
	\lesssim_{\alpha,p} \frac{\delta^\alpha [d_E^{\,-\alpha}]_{A_p^+}}{(1-t)^{\alpha+p-1}}.
	\]
	Therefore, by choosing $\delta = \delta(\alpha,p,s,t)$ small enough, we can ensure that $\sum_{i \in \Gamma \setminus \Gamma_0} (b_i-a_i) \le s|I^-|$, as required.
\end{proof}

\begin{lemma}
	\label{lem_main2}
	Let $\emptyset \ne E \subset \R$. Suppose $E$ is $(s,t)$-right median porous for some $0<s<t<1$. Then, for every $1<p<\infty$, there exists $\alpha>0$ such that $d_E^{\,-\alpha} \in A_p^+$.
\end{lemma}

\begin{proof}
	By Lemma \ref{lem_porous1}(a) and the fact that $d_E = d_{\bar{E}}$, we may assume $E$ is closed. By Lemma \ref{lem_porous1}(b), $|E| = 0$ and there exists $\delta>0$ such that $L_s(E,I^-) \ge \delta L_t(E,I^+)$ for all $I\subset\R$. The former implies that $d_E : \R \to [0,\infty)$ is a continuous function with $d_E>0$ a.e. Let $f = -\log d_E$. Then $f : \R \to (-\infty,+\infty]$ is a continuous function with $f < +\infty$ a.e.
	
	Choose $0<\sigma<\tau<\tau'<1$ such that $s<1-\tau'<1-\tau<1-\sigma=t$. We claim that there exists $C>0$ such that, for all $I\subset\R$, we have
	\begin{equation}
		\label{eq_main7}
		M_\tau(f,I^-) - M_\sigma(f,I^+) \le C.
	\end{equation}
	To prove this claim, we consider four cases.
	
	\underline{Case 1}: Suppose $I^- \cap E \ne \emptyset$ and $I^+ \cap E \ne \emptyset$. By Lemma \ref{lem_dist1}(a),
	\[
	L_s(E,I^-) \lesssim_{s,t} M_{1-\tau'}(d_E,I^-)
	\qquad \text{and} \qquad
	L_t(E,I^+) \ge M_t(d_E,I^+).
	\]
	Hence $M_{1-\tau'}(d_E,I^-) \gtrsim_{s,t,\delta} M_t(d_E,I^+)$, so there exists $C = C(s,t,\delta)$ such that
	\[
	(-\log M_{1-\tau'}(d_E,I^-)) - (-\log M_t(d_E,I^+)) \le C.
	\]
	By Lemma \ref{lem_median1}(i),
	\[
	-\log M_{1-\tau'}(d_E,I^-) \ge M_\tau(f,I^-)
	\qquad \text{and} \qquad
	-\log M_t(d_E,I^+) \le M_\sigma(f,I^+).
	\]
	Thus, \eqref{eq_main7} holds.
	
	\underline{Case 2}: Suppose $I^- \cap E = \emptyset$ and $I^+ \cap E \ne \emptyset$. Then $d(I^-,E) \le |I^-|$, so, by Lemma \ref{lem_dist1}(b),
	\[
	L_s(E,I^-) \approx_s M_s(d_E,I^-) \le M_{1-\tau'}(d_E,I^-).
	\]
	The rest of the argument is the same as in Case 1.
	
	\underline{Case 3}: Suppose $I^- \cap E \ne \emptyset$ and $I^+ \cap E = \emptyset$. Then $d(I^+,E) \le |I^+|$, so, by Lemma \ref{lem_dist1}(b),
	\[
	L_t(E,I^+) \approx_t M_t(d_E,I^+).
	\]
	The rest of the argument is the same as in Case 1.
	
	\underline{Case 4}: Suppose $I^- \cap E = \emptyset$ and $I^+ \cap E = \emptyset$. By Lemma \ref{lem_dist1}(d),
	\[
	M_t(d_E,I^+) \approx_{s,t} M_s(d_E,I^-) \le M_{1-\tau'}(d_E,I^-).
	\]
	The rest of the argument is the same as in Case 1. This completes the proof of the claim.
	
	Now, by Theorem \ref{thm_BMO}, we have $f \in \BMO^+$. Therefore, by Proposition \ref{prop_ApBMO}(h), for every $1<p<\infty$, there exists $\alpha>0$ such that $\exp(\alpha f) \in A_p^+$, i.e.\ $d_E^{\,-\alpha} \in A_p^+$.
\end{proof}

Now, we are ready to prove Theorem \ref{thm_main}.

\begin{proof}[Proof of Theorem \ref{thm_main}]
	By Lemma \ref{lem_main1}, (a) implies (d). Clearly, (d) implies (c). By Lemma \ref{lem_main2}, (c) implies (b). Clearly, (b) implies (a).
\end{proof}

\section{Consequences and examples}
\label{sec_conseq}

In this section, we look at consequences of Theorem \ref{thm_main} and present examples of right median porous sets which do not satisfy other notions of porosity. We begin with the following proposition, which ties together the theory of standard median porous sets and the theory of one-sided median porous sets.

\begin{proposition}
	\label{prop_conseq1}
	Let $E \subset \R$. Then $E$ is median porous if and only if $E$ is both left median porous and right median porous.
\end{proposition}

\begin{proof}
	The case $E = \emptyset$ is trivial, so we assume $E \ne \emptyset$. Suppose $E$ is median porous. By Theorem \ref{thm_PU1}, there exist $\alpha>0$ and $p\in(1,\infty)$ such that $d_E^{\,-\alpha} \in A_p$. By Proposition \ref{prop_ApBMO}(c), $d_E^{\,-\alpha} \in A_p^+ \cap A_p^-$. By Theorem \ref{thm_main} (and its obvious analogue for $A_p^-$ and left median porosity), $E$ is both left median porous and right median porous.
	
	Conversely, suppose $E$ is both left median porous and right median porous. By Theorem \ref{thm_main} (and its obvious analogue for $A_p^-$ and left median porosity), there exist $\alpha,\beta>0$ and $p,q \in (1,\infty)$ such that $d_E^{\,-\alpha} \in A_p^+$ and $d_E^{\,-\beta} \in A_q^-$. Let $\gamma = \min(\alpha,\beta)$ and $r = \max(p,q)$. By Proposition \ref{prop_ApBMO}(f) and (a), $d_E^{\,-\gamma} \in A_r^+ \cap A_r^-$. By Proposition \ref{prop_ApBMO}(c), $d_E^{\,-\gamma} \in A_r$. By Theorem \ref{thm_PU1}, $E$ is median porous.
\end{proof}

The following proposition (which is a slightly stronger version of \cite[Theorem 3.7]{AGGM}) can be proved similarly (using Theorems \ref{thm_ALMV1} and \ref{thm_AGGM} instead of Theorems \ref{thm_PU1} and \ref{thm_main}).

\begin{proposition}
	\label{prop_conseq2}
	Let $E \subset \R$. Then $E$ is weakly porous if and only if $E$ is both left weakly porous and right weakly porous.
\end{proposition}

The following lemma will be useful for constructing right median porous sets from median porous sets.

\begin{lemma}
	\label{lem_conseq1}
	Let $E \subset \R$ and $x \in \R$. If $E$ is $(s,t,\delta)$-right median porous, then $E \cap [x,\infty)$ and $E \cap (x,\infty)$ are $(s,t,\delta)$-right median porous.
\end{lemma}

\begin{proof}
	Let $A = E \cap [x,\infty)$ or $A = E \cap (x,\infty)$. By Lemma \ref{lem_porous1}(b), $|\bar{E}|=0$ and $L_s(E,I^-) \ge \delta L_t(E,I^+)$ for all $I\subset\R$. Since $A \subset E$, we have $|\bar{A}|=0$. By Lemma \ref{lem_porous1}(b), it suffices to show that $L_s(A,I^-) \ge \delta L_t(A,I^+)$ for all $I\subset\R$.
	
	Given $I\subset\R$, let $c$ be the midpoint of $I$. If $c \le x$, then $I^- \cap A = \emptyset$, so $L_s(A,I^-) = |I^-|$. Since $L_t(A,I^+) \le |I^+|$, it follows that $L_s(A,I^-) \ge \delta L_t(A,I^+)$. On the other hand, if $c \ge x$, then $I^+ \cap A = I^+ \cap E$, so $L_t(A,I^+) = L_t(E,I^+)$. Since $A \subset E$, we have $L_s(A,I^-) \ge L_s(E,I^-)$. Since $L_s(E,I^-) \ge \delta L_t(E,I^+)$, it follows that $L_s(A,I^-) \ge \delta L_t(A,I^+)$.
\end{proof}

The following lemma (which is a slightly stronger version of \cite[Proposition 3.8]{AGGM}) can be proved similarly (using the analogue of Lemma \ref{lem_porous1} for right weakly porous sets).

\begin{lemma}
	\label{lem_conseq2}
	Let $E \subset \R$ and $x \in \R$. If $E$ is $(s,\delta)$-right weakly porous, then $E \cap [x,\infty)$ and $E \cap (x,\infty)$ are $(s,\delta)$-right weakly porous.
\end{lemma}

Next, we give two examples of right median porous sets, one of which is not median porous and the other of which is not right weakly porous.

\begin{example}
	\label{ex1}
	It is easy to see that $\Z$ is $(\frac{1}{2},\frac{1}{2})$-weakly porous. By Proposition \ref{prop_conseq2} and Lemma \ref{lem_conseq2}, $\Z_{\ge0}$ is right weakly porous (and hence right median porous). However, $\Z_{\ge0}$ is not median porous (and hence not weakly porous). To prove this, let $I_n = (-n,n)$ and note that
	\[
	\frac{L_{1/4}(\Z_{\ge0},I_n)}{L_{1/2}(\Z_{\ge0},I_n)}
	= \frac{1}{n} \to 0 \qquad \text{as}~n \to \infty.
	\]
	Thus, there does not exist $\delta>0$ such that $L_{1/4}(\Z_{\ge0},I) \ge \delta L_{1/2}(\Z_{\ge0},I)$ for all $I\subset\R$, so $\Z_{\ge0}$ is not median porous.
\end{example}

\begin{example}
	\label{ex2}
	For any $0<\gamma<1$, let
	\[
	E_\gamma = \{\pm n^\gamma : n \in \Z_{\ge0}\}.
	\]
	It is shown in \cite[Theorem 9.1]{PU} that $E_\gamma$ is median porous (so, by Proposition \ref{prop_conseq1}, $E_\gamma$ is right median porous). However, $E_\gamma$ is not right weakly porous (and it follows by Proposition \ref{prop_conseq2} that $E_\gamma$ is not weakly porous). To prove this, let $I_n = (-2n^\gamma,2n^\gamma)$ and note that
	\[
	\frac{L_{1/2}(E_\gamma,I_n^-)}{L_1(E_\gamma,I_n^+)}
	= \frac{n^\gamma-(n-1)^\gamma}{1}
	\approx \gamma n^{\gamma-1} \to 0 \qquad \text{as}~n \to \infty.
	\]
	Thus, there does not exist $\delta>0$ such that $L_{1/2}(E_\gamma,I^-) \ge \delta L_1(E_\gamma,I^+)$ for all $I\subset\R$, so $E_\gamma$ is not right weakly porous.
\end{example}

Finally, we give an example of a right median porous set which does not satisfy any of the other notions of porosity that we have considered in this paper. Thus, right median porosity is strictly weaker than all the other properties listed in Definition \ref{def_porous}.

\begin{example}
	\label{ex3}
	For any $A \subset \R$, let $A^+ = A \cap [0,+\infty)$ and $A^- = A \cap (-\infty,0]$. Given $0<\gamma<1$ and $m \in \Z$ with $m \ge 2$, let
	\[
	E = E_\gamma^- \cup E_{\gamma/m}^+.
	\]
	We claim that $E$ is right median porous but neither median porous nor right weakly porous.
	
	First, we show that $E$ is not right weakly porous. Let $I_n = (-2n^\gamma, 2n^\gamma)$. Then
	\[
	\frac{L_{1/2}(E,I_n^-)}{L_1(E,I_n^+)}
	= \frac{n^\gamma-(n-1)^\gamma}{1}
	\approx \gamma n^{\gamma-1} \to 0 \qquad \text{as}~n \to \infty.
	\]
	Thus, there does not exist $\delta>0$ such that $L_{1/2}(E,I^-) \ge \delta L_1(E,I^+)$ for all $I\subset\R$, so $E$ is not right weakly porous.
	
	Next, we show that $E$ is not median porous. By Proposition \ref{prop_conseq1}, it suffices to show that $E$ is not left median porous. Let $I_n = (-2n^\gamma, 2n^\gamma)$. Then
	\[
	\frac{L_{1/2}(E,I_n^+)}{L_{1/2}(E,I_n^-)}
	= \frac{(n^m)^{\gamma/m}-(n^m-1)^{\gamma/m}}{n^\gamma-(n-1)^\gamma}
	\approx \frac{(\gamma/m)(n^m)^{(\gamma/m)-1}}{\gamma n^{\gamma-1}}
	= \frac{1}{mn^{m-1}} \to 0 \qquad \text{as}~n \to \infty.
	\]
	Since $L_{1/4}(E,I_n^+) \le L_{1/2}(E,I_n^+)$, it follows that there does not exist $\delta>0$ such that $L_{1/4}(E,I^+) \ge \delta L_{1/2}(E,I^-)$ for all $I\subset\R$, so $E$ is not left median porous.
	
	Finally, we show that $E$ is right median porous. Note that $E_\gamma \subset E_{\gamma/m}$ since $n^\gamma = (n^m)^{\gamma/m}$ for all $n \in \Z_{\ge0}$. Hence $E_\gamma \subset E \subset E_{\gamma/m}$, so, for all $0 < s \le 1$ and $I\subset\R$, we have
	\begin{equation}
		\label{eq_conseq1}
		L_s(E_\gamma,I) \ge L_s(E,I) \ge L_s(E_{\gamma/m},I).
	\end{equation}
	As we mentioned in Example \ref{ex2}, $E_\gamma$ and $E_{\gamma/m}$ are median porous. By Proposition \ref{prop_conseq1}, $E_\gamma$ and $E_{\gamma/m}$ are right median porous. Fix $0<s<t<1$. Then there exists $\delta>0$ such that, for all $I\subset\R$, we have
	\begin{equation}
		\label{eq_conseq2}
		L_s(E_\gamma,I^-) \ge \delta L_t(E_\gamma,I^+)
		\qquad \text{and} \qquad
		L_s(E_{\gamma/m},I^-) \ge \delta L_t(E_{\gamma/m},I^+).
	\end{equation}
	Given $I\subset\R$, let $c$ be the midpoint of $I$. If $c \ge 0$, then
	\[
	L_s(E,I^-) \ge L_s(E_{\gamma/m},I^-)
	\ge \delta L_t(E_{\gamma/m},I^+)
	= \delta L_t(E,I^+),
	\]
	where we have used \eqref{eq_conseq1}, \eqref{eq_conseq2}, and the fact that $I^+ \cap E = I^+ \cap E_{\gamma/m}$. If $c \le 0$, then
	\[
	L_s(E,I^-) = L_s(E_\gamma,I^-)
	\ge \delta L_t(E_\gamma,I^+)
	\ge \delta L_t(E,I^+),
	\]
	where we have used the fact that $I^- \cap E = I^- \cap E_\gamma$, \eqref{eq_conseq2}, and \eqref{eq_conseq1}. In either case, $L_s(E,I^-) \ge \delta L_t(E,I^+)$, so $E$ is right median porous.
\end{example}

\section{The quantitative characterization}
\label{sec_exp}

In this section, we define the one-sided Muckenhoupt exponents $\Mu_1^+(E)$ and $\Mu_\infty^+(E)$ of a set $E$ and prove Theorems \ref{thm_exp1} and \ref{thm_exp2}, which establish the precise range of exponents $\alpha>0$ such that $d_E^{\,-\alpha}$ is an $A_1^+$ weight or an $A_p^+$ weight for some $1<p<\infty$, respectively. We begin by recalling the definitions of the standard Muckenhoupt exponents $\Mu_1(E)$ and $\Mu_\infty(E)$, which were introduced in \cite{ALMV} and \cite{PU}, respectively.

\begin{definition}
	\label{def_exp1}
	Let $\emptyset \ne E \subset \R$.
	\begin{enumerate}[label=\upshape(\alph*)]
		\item If $L_1(E,I) = 0$ for some $I\subset\R$, set $\Mu_1(E)=0$. Otherwise, let $\Mu_1(E)$ be the supremum of all $\alpha \ge 0$ for which there exists $C>0$ such that, for every $I\subset\R$ centred at a point of $E$ and for every $0 < r \le L_1(E,I)$, we have
		\begin{equation*}
			\frac{|I \cap E_r|}{|I|}
			\le C \left(\frac{r}{L_1(E,I)}\right)^\alpha.
		\end{equation*}
		\item If $L_s(E,I) = 0$ for some $0<s<1$ and $I\subset\R$, set $\Mu_\infty(E)=0$. Otherwise, let $\Mu_\infty(E)$ be the supremum of all $\alpha \ge 0$ for which there exist $0<s<1$ and $C>0$ such that, for every $I\subset\R$ centred at a point of $E$ and for every $0 < r \le L_s(E,I)$, we have
		\begin{equation*}
			\frac{|I \cap E_r|}{|I|}
			\le C \left(\frac{r}{L_s(E,I)}\right)^\alpha.
		\end{equation*}
	\end{enumerate}
\end{definition}

Inspired by Definition \ref{def_exp1}, we define the one-sided Muckenhoupt exponents $\Mu_1^+(E)$ and $\Mu_\infty^+(E)$ as follows.

\begin{definition}
	\label{def_exp2}
	Let $\emptyset \ne E \subset \R$.
	\begin{enumerate}[label=\upshape(\alph*)]
		\item If $L_1(E,I) = 0$ for some $I\subset\R$, set $\Mu_1^+(E)=0$. Otherwise, let $\Mu_1^+(E)$ be the supremum of all $\alpha \ge 0$ for which there exists $C>0$ such that, for every $I\subset\R$ centred at a point of $E$ and for every $0 < r \le L_1(E,I^+)$, we have
		\begin{equation}
			\label{eq_exp1a}
			\frac{|I^- \cap E_r|}{|I^-|}
			\le C \left(\frac{r}{L_1(E,I^+)}\right)^\alpha.
		\end{equation}
		\item If $L_s(E,I) = 0$ for some $0<s<1$ and $I\subset\R$, set $\Mu_\infty^+(E)=0$. Otherwise, let $\Mu_\infty^+(E)$ be the supremum of all $\alpha \ge 0$ for which there exist $0<s<1$ and $C>0$ such that, for every $I\subset\R$ centred at a point of $E$ and for every $0 < r \le L_s(E,I^+)$, we have
		\begin{equation}
			\label{eq_exp1b}
			\frac{|I^- \cap E_r|}{|I^-|}
			\le C \left(\frac{r}{L_s(E,I^+)}\right)^\alpha.
		\end{equation}
	\end{enumerate}
\end{definition}

The quantities $\Mu_1^-(E)$ and $\Mu_\infty^-(E)$ are defined similarly; simply interchange $I^-$ and $I^+$ in the definitions of $\Mu_1^+(E)$ and $\Mu_\infty^+(E)$.

The following lemma provides some preliminary information about the possible values of $\Mu_1^+(E)$ and $\Mu_\infty^+(E)$.

\begin{lemma}
	\label{lem_exp1}
	Let $\emptyset \ne E \subset \R$.
	\begin{enumerate}[label=\upshape(\alph*)]
		\item $0 \le \Mu_1^+(E) \le 1$.
		\item $0 \le \Mu_\infty^+(E) \le 1$.
		\item Let $\alpha>0$. If $d_E^{\,-\alpha} \in L^1_\mathrm{loc}$, then $\alpha<1$.
	\end{enumerate}
\end{lemma}

\begin{proof}
	(a) Assume $L_1(E,I)>0$ for all $I\subset\R$. Note that \eqref{eq_exp1a} is satisfied if $\alpha=0$ and $C=1$. Thus, $\Mu_1^+(E) \ge 0$. Now, suppose $\alpha \ge 0$ and $C>0$ are such that \eqref{eq_exp1a} holds. Pick $x \in E$ and pick $I\subset\R$ centred at $x$. Then, for all $0 < r \le |I^-|$, we have $I^- \cap E_r \supset (x-r,x)$ and hence $|I^- \cap E_r| \ge r$. Together with \eqref{eq_exp1a}, this implies that, for all $0 < r \le L_1(E,I^+)$, we have $r \le K r^\alpha$, where $K$ is a constant that does not depend on $r$. If $\alpha>1$, then $1 \le K r^{\alpha-1} \to 0$ as $r \to 0$, which is a contradiction. Therefore, we must have $\alpha \le 1$, so $\Mu_1^+(E) \le 1$.
	
	(b) This is similar to (a), so we omit the details.
	
	(c) Pick $x \in E$. Then $d_E(y) \le |y-x|$ for all $x \in \R$, so
	\[
	\infty > \int_{x-1}^{x+1} d_E^{\,-\alpha}
	\ge \int_{x-1}^{x+1} |y-x|^{-\alpha}\,dy
	= \int_{-1}^1 |y|^{-\alpha}\,dy
	\]
	and hence $\alpha<1$.
\end{proof}

\begin{remark}
	\label{rmk_exp}
	In Lemma \ref{lem_exp1}, the converse of (c) is false, even if $|\bar{E}|=0$. For example, let $E = \{1/n : n \in \Z_{>0}\} \cup \{0\}$. Then $E$ is a countable compact set, so $|\bar{E}|=0$. However, for any $n \in \Z_{>0}$ and $x \in (\frac{1}{n+1},\frac{1}{n})$, we have $d_E(x) = \min(x-\frac{1}{n+1},\frac{1}{n}-x)$, so, for all $1/2 \le \alpha < 1$, we have
	\begin{align*}
		\int_0^1 d_E^{\,-\alpha}
		&= \sum_{n=1}^\infty \int_{1/(n+1)}^{1/n} d_E^{\,-\alpha}
		= 2 \sum_{n=1}^\infty \int_{0}^{[(1/n)-(1/(n+1))]/2} x^{-\alpha}\,dx \\
		&\approx_\alpha \sum_{n=1}^\infty \left(\frac{1}{n}-\frac{1}{n+1}\right)^{1-\alpha}
		\approx_\alpha \sum_{n=1}^\infty \frac{1}{n^{2(1-\alpha)}}
		= \infty.
	\end{align*}
\end{remark}

The following two lemmas will be used in the proofs of Theorems \ref{thm_exp1} and \ref{thm_exp2}.

\begin{lemma}
	\label{lem_dist2}
	Let $\emptyset \ne E \subset \R$, $0<\alpha<1$, $c>0$, and $I\subset\R$.
	\begin{enumerate}[label=\upshape(\alph*)]
		\item If $I \cap E = \emptyset$, then $\dashint_I d_E^{\,-\alpha} \lesssim_\alpha |I|^{-\alpha}$ and $\esssup_I d_E^{\,\alpha} \gtrsim_\alpha |I|^\alpha$.
		\item If $d(I,E) \le c|I|$, then $\dashint_I d_E^{\,-\alpha} \gtrsim_{\alpha,c} |I|^{-\alpha}$ and $\esssup_I d_E^{\,\alpha} \lesssim_{\alpha,c} |I|^\alpha$.
		\item If $d(I,E) \ge c|I|$, then $\dashint_I d_E^{\,-\alpha} \approx_{\alpha,c} d(I,E)^{-\alpha}$ and $\esssup_I d_E^{\,\alpha} \approx_{\alpha,c} d(I,E)^\alpha$.
		\item $\esssup_I d_E^{\,\alpha} \gtrsim_\alpha L_1(E,I)^\alpha$.
		\item If $\bar{I} \cap \bar{E} \ne \emptyset$, then $\esssup_I d_E^{\,\alpha} \le L_1(E,I)^\alpha$.
	\end{enumerate}
\end{lemma}

\begin{proof}
	(a) Write $I = (a,b)$. Then $d_E(x) \ge \min(x-a,b-x)$ for all $x \in I$, so
	\[
	\int_a^b d_E^{\,-\alpha}
	\le \int_a^b \min(x-a,b-x)^{-\alpha}\,dx
	= 2 \int_0^{(b-a)/2} x^{-\alpha}\,dx
	= \frac{2}{1-\alpha} \left(\frac{b-a}{2}\right)^{1-\alpha},
	\]
	and hence $\frac{1}{b-a} \int_a^b d_E^{\,-\alpha} \le \frac{2^\alpha}{1-\alpha} (b-a)^{-\alpha}$. Let $c$ be the midpoint of $I$. Then $d_E(c) \ge \frac{1}{2}|I|$. Since $d_E$ is continuous, it follows that $\esssup_I d_E^{\,\alpha} = \sup_I d_E^{\,\alpha} \ge 2^{-\alpha} |I|^\alpha$.
	
	(b) For all $x \in I$, we have $d_E(x) \le d(I,E) + |I| \le (c+1)|I|$. Hence $\dashint_I d_E^{\,-\alpha} \ge (c+1)^{-\alpha} |I|^{-\alpha}$ and $\esssup_I d_E^{\,\alpha} \le (c+1)^\alpha |I|^\alpha$.
	
	(c) For all $x \in I$, we have $d(I,E) \le d_E(x) \le d(I,E) + |I| \le (1+c^{-1})d(I,E)$. Hence $d(I,E)^{-\alpha} \ge \dashint_I d_E^{\,-\alpha} \ge (1+c^{-1})^{-\alpha} d(I,E)^{-\alpha}$ and $d(I,E)^\alpha \le \essinf_I d_E^{\,\alpha} \le \esssup_I d_E^{\,\alpha} \le (1+c^{-1})^\alpha d(I,E)^\alpha$.
	
	(d) Assume $L_1(E,I)>0$. Let $0<\ell<L_1(E,I)$. Then there exists $J \subset I$ such that $J \cap E = \emptyset$ and $|J| > \ell$. Let $x$ be the midpoint of $J$. Then $d_E(x) \ge \frac{1}{2}|J| > \frac{1}{2}\ell$. Since $d_E$ is continuous, it follows that $\esssup_I d_E^{\,\alpha} = \sup_I d_E^{\,\alpha} > 2^{-\alpha} \ell^\alpha$, so $\esssup_I d_E^{\,\alpha} \ge 2^{-\alpha} L_1(E,I)^\alpha$.
	
	(e) Suppose $\esssup_I d_E^{\,\alpha} > L_1(E,I)^\alpha$. Then, since $d_E$ is continuous, we have $\sup_I d_E^{\,\alpha} > L_1(E,I)^\alpha$, so there exists $x \in I$ such that $d_E(x)^\alpha > L_1(E,I)^\alpha$, i.e. $d_E(x) > L_1(E,I)$. Pick $\ell$ such that $d_E(x) > \ell > L_1(E,I)$. Then, since $d_E(x) > \ell$, we have $(x-\ell,x+\ell) \cap E = \emptyset$. But $\ell > L_1(E,I)$, so $(x-\ell,x) \not\subset I$ and $(x,x+\ell) \not\subset I$. Since $x \in I$, this implies that $\bar{I} \subset (x-\ell,x+\ell)$. But then $\bar{I} \cap \bar{E} = \emptyset$, which is a contradiction.
\end{proof}

\begin{lemma}
	\label{lem_exp2}
	Let $\emptyset \ne E \subset \R$ and let $\alpha>0$. Suppose $\alpha < \Mu_1^+(E)$ (resp.\ $\alpha < \Mu_\infty^+(E)$). Then $|\bar{E}|=0$ and $d_E^{\,-\alpha} \in L^1_\mathrm{loc}$. Moreover, for $s=1$ (resp.\ for some $0<s<1$), there exists $K>0$ such that, for every $I\subset\R$ centred at a point of $E$, we have
	\begin{equation}
		\label{eq_exp2a}
		\dashint_{I^-} d_E^{\,-\alpha} \le K L_s(E,I^+)^{-\alpha}.
	\end{equation}
\end{lemma}

\begin{proof}
	There exists $\beta>\alpha$ such that, for $s=1$ (resp.\ for some $0<s<1$) and some $C>0$, every $I\subset\R$ satisfies $L_s(E,I)>0$, and every $I\subset\R$ centred at a point of $E$ satisfies
	\begin{equation}
		\label{eq_exp2b}
		\frac{|I^- \cap E_r|}{|I^-|}
		\le C \left(\frac{r}{L_s(E,I^+)}\right)^\beta
		\qquad \text{for all} ~ 0 < r \le L_s(E,I^+).
	\end{equation}
	
	First, we show that $|\bar{E}|=0$. Suppose $I\subset\R$ is centred at a point of $E$. Since $\bar{E} \subset E_r$ for all $r>0$, we may replace $E_r$ with $\bar{E}$ in \eqref{eq_exp2b}. Then, by letting $r \to 0$, we get $|I^- \cap \bar{E}|=0$. Thus, we have shown that
	\begin{equation}
		\label{eq_exp2c}
		|(a,b) \cap \bar{E}| = 0
		\qquad \text{for all}~b \in E~\text{and}~a<b.
	\end{equation}
	
	\underline{Case 1}: Suppose $\sup E = x < \infty$. For each $n \in \Z_{>0}$, there exists $x_n \in E$ such that $x_n > x - \frac{1}{n}$. Let $I_n = (x-n,x_n)$. By \eqref{eq_exp2c}, we have $|I_n \cap \bar{E}|=0$. Since $\bigcup_{n=1}^\infty I_n = (-\infty,x)$, this implies that $|(-\infty,x) \cap \bar{E}|=0$. Since $\bar{E} \subset (-\infty,x]$, it follows that $|\bar{E}|=0$.
	
	\underline{Case 2}: Suppose $\sup E = \infty$. For each $n \in \Z_{>0}$, there exists $x_n \in E$ such that $x_n > n$. Let $I_n = (-n,x_n)$. By \eqref{eq_exp2c}, we have $|I_n \cap \bar{E}|=0$. Since $\bigcup_{n=1}^\infty I_n = \R$, it follows that $|\bar{E}|=0$.
	
	Now, we show that $d_E^{\,-\alpha} \in L^1_\mathrm{loc}$. Suppose $I\subset\R$ is centred at a point of $E$. Since $L_s(E,I^+) > 0$, there is a unique $j \in \Z$ such that
	\begin{equation}
		\label{eq_exp2d}
		2^j \le L_s(E,I^+) < 2^{j+1}.
	\end{equation}
	Since $\bar{E} = \{d_E=0\}$, we have
	\[
	\R = \bar{E} \sqcup \Bigl(\bigsqcup_{-\infty < i \le j} (E_{2^i} \setminus E_{2^{i-1}})\Bigr) \sqcup (\R \setminus E_{2^j}).
	\]
	Since $|\bar{E}|=0$, it follows that
	\begin{align}
		\label{eq_exp2e}
		\int_{I^-} d_E^{\,-\alpha}
		&= \sum_{-\infty < i \le j} \int_{I^- \cap (E_{2^i} \setminus E_{2^{i-1}})} d_E^{\,-\alpha} + \int_{I^- \setminus E_{2^j}} d_E^{\,-\alpha} \\
		\nonumber &\le \sum_{-\infty < i \le j} (2^{i-1})^{-\alpha} |I^- \cap E_{2^i}| + (2^j)^{-\alpha} |I^-| \\
		\nonumber &\le |I^-| \Bigl(CL_s(E,I^+)^{-\beta} \sum_{-\infty < i \le j} (2^{i-1})^{-\alpha} (2^i)^\beta + (2^j)^{-\alpha}\Bigr) \\
		\nonumber &= 2^\alpha |I^-| \left(C L_s(E,I^+)^{-\beta} \frac{(2^j)^{\beta-\alpha}}{1-2^{\alpha-\beta}} + (2^{j+1})^{-\alpha}\right) \\
		\nonumber &\le 2^\alpha |I^-| L_s(E,I^+)^{-\alpha} \left(\frac{C}{1-2^{\alpha-\beta}}+1\right).
	\end{align}
	Here the first inequality follows from the definition of $E_r$, the second inequality follows from \eqref{eq_exp2b}, and the third inequality follows from \eqref{eq_exp2d}. In particular, we have shown that \eqref{eq_exp2a} holds, with a constant $K$ depending only on $\alpha$, $\beta$ and $C$.
	
	\underline{Case 1}: Suppose $\sup E = x < \infty$. For each $n \in \Z_{>0}$, there exists $x_n \in E$ such that $x_n > x - 2^{-n-1}$. We may assume $x_1 \le x_2 \le \cdots$. Given $k \in \Z_{>0}$, let $I_k = (x-k, x+k)$ and $I_{k,n} = (x_n - (k-2^{-n}), x_n + (k-2^{-n}))$. It is easy to check that $I_{k,n}^- \uparrow I_k^-$ as $n \to \infty$. For all sufficiently large $n$, we have $|I_{k,n}^+ \cap (x,\infty)| \ge \frac{k}{2}$ and $|I_{k,n}^+ \cap (x,\infty)| \ge (1-s)|I_{k,n}^+|$, so $L_s(E,I_{k,n}^+) \ge \frac{k}{2}$. Plugging this and $|I_{k,n}^-| \le k$ into \eqref{eq_exp2e}, we get
	\[
	\int_{I_{k,n}^-} d_E^{\,-\alpha} \lesssim_{\alpha,\beta,C,k} 1.
	\]
	Letting $n \to \infty$, we get
	\[
	\int_{I_k^-} d_E^{\,-\alpha} \lesssim_{\alpha,\beta,C,k} 1.
	\]
	In particular, $\int_{x-k}^x d_E^{\,-\alpha} < \infty$. On the other hand, $d_E(y) \ge y-x$ for all $y>x$, and $\alpha<1$ by Lemma \ref{lem_exp1}(a) or (b), so
	\[
	\int_x^{x+k} d_E^{\,-\alpha}
	\le \int_x^{x+k} (y-x)^{-\alpha} \,dy
	= \int_0^k y^{-\alpha} \,dy
	<\infty.
	\]
	Thus, $\int_{x-k}^{x+k} d_E^{\,-\alpha} < \infty$ for all $k \in \Z_{>0}$.
	
	\underline{Case 2}: Suppose $\sup E = \infty$. For each $n \in \Z_{>0}$, there exists $x_n \in E$ such that $x_n > n$. Note that $\int_{-n}^n d_E^{\,-\alpha} \le \int_{-n}^{x_n} d_E^{\,-\alpha} < \infty$ by \eqref{eq_exp2e}. Thus, $\int_{-n}^n d_E^{\,-\alpha} < \infty$ for all $n \in \Z_{>0}$.
\end{proof}

Now, we are ready to prove Theorems \ref{thm_exp1} and \ref{thm_exp2}.

\begin{proof}[Proof of Theorem \ref{thm_exp1}]
	($\implies$) Suppose $d_E^{\,-\alpha} \in A_1^+$. In particular, $d_E^{\,-\alpha} \in L^1_\mathrm{loc}$, so $\alpha<1$ by Lemma \ref{lem_exp1}(c). Moreover, $d_E>0$ a.e., so $|\bar{E}|=0$. By Lemma \ref{lem_length2}, $L_1(E,I)>0$ for all $I\subset\R$. Suppose $I\subset\R$ is centred at a point of $E$ and let $r>0$ be given. Then
	\[
	\dashint_{I^-} d_E^{\,-\alpha}
	\ge \frac{1}{|I^-|} \int_{I^- \cap E_r} d_E^{\,-\alpha}
	\ge \frac{|I^- \cap E_r|}{|I^-|} r^{-\alpha}.
	\]
	By Lemma \ref{lem_dist2}(d),
	\[
	\esssup_{I^+} d_E^{\,\alpha} \gtrsim_\alpha L_1(E,I^+)^\alpha.
	\]
	Hence
	\[
	\frac{|I^- \cap E_r|}{|I^-|} \left(\frac{L_1(E,I^+)}{r}\right)^\alpha
	\lesssim_\alpha [d_E^{\,-\alpha}]_{A_1^+}.
	\]
	Thus, $\alpha \le \Mu_1^+(E)$. By Proposition \ref{prop_ApBMO}(g), there exists $\delta>0$ such that $d_E^{\,-\alpha(1+\delta)} \in A_1^+$. By the same reasoning, $\alpha(1+\delta) \le \Mu_1^+(E)$, so $\alpha < \Mu_1^+(E)$.
	
	($\impliedby$) Suppose $\alpha < \Mu_1^+(E)$. By Lemma \ref{lem_exp1}(a), $\alpha<1$. By Lemma \ref{lem_exp2}, $d_E^{\,-\alpha} \in L^1_\mathrm{loc}$. Since $d_E < \infty$ on $\R$, we have $d_E^{\,-\alpha} > 0$ on $\R$, so $d_E^{\,-\alpha}$ is a weight. We want to show that $d_E^{\,-\alpha} \in A_1^+$, i.e., for all $I\subset\R$,
	\begin{equation}
		\label{eq_exp3a}
		\left(\dashint_{I^-} d_E^{\,-\alpha}\right) \esssup_{I^+} d_E^{\,\alpha}
		\lesssim 1,
	\end{equation}
	with a constant independent of $I$. Given $I\subset\R$, write $I^- = (a,b)$ and $I^+ = (b,c)$. We consider three cases.
	
	\underline{Case 1}: Suppose $b \in E$. By Lemma \ref{lem_exp2},
	\[
	\dashint_{I^-} d_E^{\,-\alpha}
	\lesssim_\alpha L_1(E,I^+)^{-\alpha}.
	\]
	By Lemma \ref{lem_dist2}(e),
	\[
	\esssup_{I^+} d_E^{\,\alpha}
	\le L_1(E,I^+)^\alpha.
	\]
	Hence \eqref{eq_exp3a} holds.
	
	\underline{Case 2}: Suppose $I^- \cap E \ne \emptyset$. Then $\overline{I^- \cap E}$ is a nonempty compact set, so it has a maximum $x$. Note that $a < x \le b$ and let $J$ be the unique interval such that $J^- = (a,x)$. For every $\epsilon>0$, there exists $y \in I^- \cap E$ such that $x-\epsilon < y \le x$. Let $J_y$ be the unique interval such that $J_y^- = (a,y)$. By Case 1, \eqref{eq_exp3a} holds with $J_y$ in place of $I$. Letting $\epsilon \to 0$ (so that $y \to x$) and using the fact that $d_E^{\,-\alpha} \in L^1_\mathrm{loc}$ and $d_E^{\,\alpha}$ is continuous, we get
	\begin{equation}
		\label{eq_exp3b}
		\left(\dashint_{J^-} d_E^{\,-\alpha}\right) \esssup_{J^+} d_E^{\,\alpha}
		\lesssim 1.
	\end{equation}
	If $x = b$ (so that $J = I$), we are done, so we assume $x < b$. Note that $(x,b) \cap E = \emptyset$ by our choice of $x$. Write $J^+ = (x,y)$ and note that $y<c$.
	
	\underline{Subcase 2a}: Suppose $J^+ \subset I^-$ (i.e.\ $y \le b$). By Lemma \ref{lem_dist2},
	\begin{enumerate}[label=\upshape(\roman*)]
		\item $\dashint_{(x,b)} d_E^{\,-\alpha} \lesssim_\alpha (b-x)^{-\alpha}$ (since $(x,b) \cap E = \emptyset$);
		\item $\dashint_{J^-} d_E^{\,-\alpha} \gtrsim_\alpha |J^-|^{-\alpha}$ (since $x \in \bar{E}$);
		\item $\esssup_{J^+} d_E^{\,\alpha} \gtrsim_\alpha |J^+|^\alpha$ (since $J^+ \cap E = \emptyset$);
		\item $\esssup_{I^+} d_E^{\,\alpha} \lesssim_\alpha |I^+|^\alpha$ (since $I^- \cap E \ne \emptyset$).
	\end{enumerate}
	We estimate
	\begin{align*}
		\tint_{I^-} d_E^{\,-\alpha}
		&= \tint_{J^-} d_E^{\,-\alpha} + \tint_x^b d_E^{\,-\alpha} \\
		&\lesssim \tint_{J^-} d_E^{\,-\alpha} + (b-x)^{1-\alpha} \\
		&\le \tint_{J^-} d_E^{\,-\alpha} + |I^-|^{1-\alpha} \\
		&\lesssim \tint_{J^-} d_E^{\,-\alpha} + |I^-|^{1-\alpha} |J^-|^{\alpha-1} \tint_{J^-} d_E^{\,-\alpha} \\
		&\lesssim (|I^-|/|J^-|)^{1-\alpha} \tint_{J^-} d_E^{\,-\alpha},
	\end{align*}
	where the first inequality follows from (i), the second inequality follows from $b-x \le |I^-|$, the third inequality follows from (ii), and the fourth inequality follows from $|J^-| \le |I^-|$. Hence
	\[
	\tdashint_{I^-} d_E^{\,-\alpha}
	\lesssim (|J^-|/|I^-|)^\alpha \, \dashint_{J^-} d_E^{\,-\alpha}.
	\]
	On the other hand, by (iv) and (iii), we have
	\[
	\tesssup_{I^+} d_E^{\,\alpha}
	\lesssim (|I^+|/|J^+|)^\alpha \tesssup_{J^+} d_E^{\,\alpha}.
	\]
	Together with \eqref{eq_exp3b}, the last two inequalities imply \eqref{eq_exp3a}.
	
	\underline{Subcase 2b}: Suppose $J^+ \not\subset I^-$ (i.e.\ $y>b$). By Lemma \ref{lem_dist2},
	\begin{enumerate}[label=\upshape(\roman*)]
		\item $\dashint_{(x,b)} d_E^{\,-\alpha} \lesssim_\alpha (b-x)^{-\alpha}$ (since $(x,b) \cap E = \emptyset$);
		\item $\esssup_{J^+} d_E^{\,\alpha} \gtrsim_\alpha L_1(E,J^+)^\alpha$;
		\item $\esssup_{J^+} d_E^{\,\alpha} \lesssim_\alpha |J^+|^\alpha$ (since $x \in \bar{E}$);
		\item $\esssup_{I^+} d_E^{\,\alpha} \lesssim_\alpha |I^+|^\alpha$ (since $I^- \cap E \ne \emptyset$).
	\end{enumerate}
	Also note that
	\begin{enumerate}[label=\upshape(\roman*), resume]
		\item $b-x \le L_1(E,J^+)$ (since $(x,b) \subset J^+ \setminus E$).
	\end{enumerate}
	We estimate
	\begin{align*}
		\tint_{I^-} d_E^{\,-\alpha}
		&= \tint_{J^-} d_E^{\,-\alpha} + \tint_x^b d_E^{\,-\alpha} \\
		&\lesssim \tint_{J^-} d_E^{\,-\alpha} + (b-x)^{1-\alpha} \\
		&\le \tint_{J^-} d_E^{\,-\alpha} + |I^-|^{1-\alpha},
	\end{align*}
	where the first inequality follows from (i) and the second inequality follows from $b-x \le |I^-|$. Since $|J^-| \le |I^-| \le 2|J^-|$, it follows that
	\begin{equation}
		\label{eq_exp3c}
		\tdashint_{I^-} d_E^{\,-\alpha}
		\lesssim \tdashint_{J^-} d_E^{\,-\alpha} + |I^-|^{-\alpha}.
	\end{equation}
	
	First, suppose $b-x \ge \frac{1}{4}|I^-|$. Then
	\begin{align*}
		(\tdashint_{I^-} d_E^{\,-\alpha}) \tesssup_{I^+} d_E^{\,\alpha}
		&\lesssim (\tdashint_{J^-} d_E^{\,-\alpha} + |I^-|^{-\alpha}) |I^+|^\alpha \\
		&= (\tdashint_{J^-} d_E^{\,-\alpha}) |I^+|^\alpha + 1 \\
		&\lesssim (\tdashint_{J^-} d_E^{\,-\alpha}) L_1(E,J^+)^\alpha + 1 \\
		&\lesssim (\tdashint_{J^-} d_E^{\,-\alpha}) \tesssup_{J^+} d_E^{\,\alpha} + 1 \\
		&\lesssim 1,
	\end{align*}
	where the first inequality follows from \eqref{eq_exp3c} and (iv), the second inequality follows from $|I^-| \lesssim b-x$ and (v), the third inequality follows from (ii), and the fourth inequality follows from \eqref{eq_exp3b}.
	
	Now, suppose $b-x \le \frac{1}{4}|I^-|$. Then, since $d_E$ is $1$-Lipschitz, we have
	\[
	\tesssup_{I^+} d_E \le \tesssup_{J^+} d_E + (c-y),
	\]
	so
	\begin{align}
		\label{eq_exp3d}
		\tesssup_{I^+} d_E^{\,\alpha}
		&\le \tesssup_{J^+} d_E^{\,\alpha} + (c-y)^\alpha \\
		\nonumber &\lesssim \tesssup_{J^+} d_E^{\,\alpha} + L_1(E,J^+)^\alpha \\
		\nonumber &\lesssim \tesssup_{J^+} d_E^{\,\alpha},
	\end{align}
	where the first inequality follows from the fact that $(\xi + \eta)^\alpha \le \xi^\alpha + \eta^\alpha$ for all $\xi,\eta\ge0$ (since $0<\alpha<1$), the second inequality follows from $c-y=2(b-x)$ and (v), and the third inequality follows from (ii). Consequently,
	\begin{align*}
		(\tdashint_{I^-} d_E^{\,-\alpha}) \tesssup_{I^+} d_E^{\,\alpha}
		&\lesssim (\tdashint_{J^-} d_E^{\,-\alpha} + |I^-|^{-\alpha}) \tesssup_{J^+} d_E^{\,\alpha} \\
		&\lesssim 1 + |I^-|^{-\alpha} \tesssup_{J^+} d_E^{\,\alpha} \\
		&\lesssim 1 + |I^-|^{-\alpha} |I^+|^\alpha \\
		&\lesssim 1,
	\end{align*}
	where the first inequality follows from \eqref{eq_exp3c} and \eqref{eq_exp3d}, the second inequality follows from \eqref{eq_exp3b}, and the third inequality follows from (iii) and $|J^+| \le |I^+|$.
	
	\underline{Case 3}: Suppose $I^- \cap E = \emptyset$.
	
	\underline{Subcase 3a}: Suppose $d(I^+,E) \le 2|I^+|$. By Lemma \ref{lem_dist2}, we have $\dashint_{I^-} d_E^{\,-\alpha} \lesssim_\alpha |I^-|^{-\alpha}$ and $\esssup_{I^+} d_E^{\,\alpha} \lesssim_\alpha |I^+|^\alpha$, so \eqref{eq_exp3a} holds.
	
	\underline{Subcase 3b}: Suppose $d(I^+,E) \ge 2|I^+|$. Since $|d(I^-,E)-d(I^+,E)| \le \frac{1}{2}|I|$, we have $d(I^-,E) \ge |I^-|$. By Lemma \ref{lem_dist2}, $\dashint_{I^-} d_E^{\,-\alpha} \approx_\alpha d(I^-,E)^{-\alpha}$ and $\esssup_{I^+} d_E^{\,\alpha} \approx_\alpha d(I^+,E)^\alpha$. Moreover, $d(I^-,E) \ge \frac{1}{2} d(I^+,E)$, so $d(I^-,E)^{-\alpha} \lesssim d(I^+,E)^{-\alpha}$. Hence \eqref{eq_exp3a} holds.
\end{proof}

\begin{proof}[Proof of Theorem \ref{thm_exp2}]
	($\implies$) Suppose $d_E^{\,-\alpha} \in A_p^+$ for some $1<p<\infty$. In particular, $d_E^{\,-\alpha} \in L^1_\mathrm{loc}$, so $\alpha<1$ by Lemma \ref{lem_exp1}(c). Moreover, $d_E>0$ a.e., so $|\bar{E}|=0$. Fix $0<s<1$. By Lemma \ref{lem_length2}, $L_s(E,I) > 0$ for all $I\subset\R$. Suppose $I\subset\R$ is centred at a point of $E$ and let $r>0$ be given. Then
	\[
	\dashint_{I^-} d_E^{\,-\alpha}
	\ge \frac{1}{|I^-|} \int_{I^- \cap E_r} d_E^{\,-\alpha}
	\ge \frac{|I^- \cap E_r|}{|I^-|} r^{-\alpha}.
	\]
	By Lemma \ref{lem_dist1}(c),
	\[
	L_s(E,I^+)^\alpha
	\lesssim_{\alpha,p,s} \left(\dashint_{I^+} d_E^{\,\alpha/(p-1)}\right)^{p-1}.
	\]
	Hence
	\[
	\frac{|I^- \cap E_r|}{|I^-|} \left(\frac{L_s(E,I^+)}{r}\right)^\alpha
	\lesssim_{\alpha,p,s} [d_E^{\,-\alpha}]_{A_p^+}.
	\]
	Thus, $\alpha \le \Mu_\infty^+(E)$. By Proposition \ref{prop_ApBMO}(a), (g) and (b), there exist $\delta>0$ and $1<q<\infty$ such that $d_E^{\,-\alpha(1+\delta)} \in A_q^+$. By the same reasoning, $\alpha(1+\delta) \le \Mu_\infty^+(E)$, so $\alpha < \Mu_\infty^+(E)$.
	
	($\impliedby$) Suppose $\alpha < \Mu_\infty^+(E)$. Fix $\beta$ such that $\alpha < \beta < \Mu_\infty^+(E)$. By Lemma \ref{lem_exp1}(b), $\beta<1$. By Lemma \ref{lem_exp2}, $|\bar{E}|=0$ and $d_E^{\,-\alpha}, d_E^{\,-\beta} \in L^1_\mathrm{loc}$. Since $d_E < \infty$ on $\R$, we have $d_E^{\,-\alpha} > 0$ on $\R$, so $d_E^{\,-\alpha}$ is a weight. We want to show that $d_E^{\,-\alpha} \in A_p^+$ for some $1<p<\infty$. By Proposition \ref{prop_ApBMO}(e) and (b), it suffices to show that, for all $I\subset\R$,
	\begin{equation}
		\label{eq_exp4a}
		\dashint_{I^-} d_E^{\,-\beta}
		\lesssim \left(\dashint_{I^+} d_E^{\,-\alpha}\right)^{\beta/\alpha},
	\end{equation}
	with a constant independent of $I$. Given $I\subset\R$, write $I^- = (a,b)$ and $I^+ = (b,c)$. We consider three cases.
	
	\underline{Case 1}: Suppose $b \in E$. By Lemma \ref{lem_exp2}, there exists $0<s<1$ such that
	\[
	\dashint_{I^-} d_E^{\,-\beta}
	\lesssim_{\beta,s} L_s(E,I^+)^{-\beta}.
	\]
	By Lemma \ref{lem_dist1}(c),
	\[
	L_s(E,I^+)^{-\alpha}
	\lesssim_{\alpha,s} \dashint_{I^+} d_E^{\,-\alpha}.
	\]
	Hence \eqref{eq_exp4a} holds.
	
	\underline{Case 2}: Suppose $I^- \cap E \ne \emptyset$. Then $\overline{I^- \cap E}$ is a nonempty compact set, so it has a maximum $x$. Note that $a < x \le b$ and let $J$ be the unique interval such that $J^- = (a,x)$. For every $\epsilon>0$, there exists $y \in I^- \cap E$ such that $x-\epsilon < y \le x$. Let $J_y$ be the unique interval such that $J_y^- = (a,y)$. By Case 1, \eqref{eq_exp4a} holds with $J_y$ in place of $I$. Letting $\epsilon \to 0$ (so that $y \to x$) and using the fact that $d_E^{\,-\alpha}, d_E^{\,-\beta} \in L^1_\mathrm{loc}$, we get
	\begin{equation}
		\label{eq_exp4b}
		\dashint_{J^-} d_E^{\,-\beta}
		\lesssim \left(\dashint_{J^+} d_E^{\,-\alpha}\right)^{\beta/\alpha}.
	\end{equation}
	If $x = b$ (so that $J = I$), we are done, so we assume $x < b$. Note that $(x,b) \cap E = \emptyset$ by our choice of $x$. Write $J^+ = (x,y)$ and note that $y<c$.
	
	\underline{Subcase 2a}: Suppose $J^+ \subset I^-$ (i.e.\ $y \le b$). By Lemma \ref{lem_dist2},
	\begin{enumerate}[label=\upshape(\roman*)]
		\item $\dashint_{(x,b)} d_E^{\,-\beta} \lesssim_\beta (b-x)^{-\beta}$ (since $(x,b) \cap E = \emptyset$);
		\item $\dashint_{J^+} d_E^{\,-\alpha} \lesssim_\alpha |J^+|^{-\alpha}$ (since $J^+ \cap E = \emptyset$);
		\item $\dashint_{I^+} d_E^{\,-\alpha} \gtrsim_\alpha |I^+|^{-\alpha}$ (since $I^- \cap E \ne \emptyset$).
	\end{enumerate}
	Hence
	\begin{align*}
		\tint_{I^-} d_E^{\,-\beta}
		&= \tint_{J^-} d_E^{\,-\beta} + \int_x^b d_E^{\,-\beta} \\
		&\lesssim |J^-| (\tdashint_{J^+} d_E^{\,-\alpha})^{\beta/\alpha} + (b-x)^{1-\beta} \\
		&\lesssim |J^-| |J^+|^{-\beta} + |I^-|^{1-\beta} \\
		&\lesssim |I^-| |I^+|^{-\beta} \\
		&\lesssim |I^-| (\tdashint_{I^+} d_E^{\,-\alpha})^{\beta/\alpha}.
	\end{align*}
	where the first inequality follows from \eqref{eq_exp4b} and (i), the second inequality follows from (ii) and $b-x \le |I^-|$, the third inequality follows from $|J^-| \le |I^-|$, and the fourth inequality follows from (iii). Thus, \eqref{eq_exp4a} holds.
	
	\underline{Subcase 2b}: Suppose $J^+ \not\subset I^-$ (i.e. $y>b$). By Lemma \ref{lem_dist2},
	\begin{enumerate}[label=\upshape(\roman*)]
		\item $\dashint_{(x,b)} d_E^{\,-\alpha} \lesssim_\alpha (b-x)^{-\alpha}$ and $\dashint_{(x,b)} d_E^{\,-\beta} \lesssim_\beta (b-x)^{-\beta}$ (since $(x,b) \cap E = \emptyset$);
		\item $\dashint_{J^-} d_E^{\,-\beta} \gtrsim_\beta |J^-|^{-\beta}$ (since $x \in \bar{E}$);
		\item $\dashint_{I^+} d_E^{\,-\alpha} \gtrsim_\alpha |I^+|^{-\alpha}$ (since $I^- \cap E \ne \emptyset$).
	\end{enumerate}
	We estimate
	\begin{align*}
		\tint_{I^-} d_E^{\,-\beta}
		&= \tint_{J^-} d_E^{\,-\beta} + \tint_x^b d_E^{\,-\beta} \\
		&\lesssim \tint_{J^-} d_E^{\,-\beta} + (b-x)^{1-\beta} \\
		&\le \tint_{J^-} d_E^{\,-\beta} + |J^-|^{1-\beta} \\
		&\lesssim \tint_{J^-} d_E^{\,-\beta},
	\end{align*}
	where the first inequality follows from (i), the second inequality follows from $b-x \le |J^+|$, and the third inequality follows from (ii). On the other hand,
	\begin{align*}
		\tint_{J^+} d_E^{\,-\alpha}
		&= \tint_x^b d_E^{\,-\alpha} + \tint_b^y d_E^{\,-\alpha} \\
		&\lesssim (b-x)^{1-\alpha} + \tint_b^y d_E^{\,-\alpha} \\
		&\le |I^+|^{1-\alpha} + \tint_{I^+} d_E^{\,-\alpha} \\
		&\lesssim \tint_{I^+} d_E^{\,-\alpha},
	\end{align*}
	where the first inequality follows from (i), the second inequality follows from $b-x \le |I^-|$ and $(b,y) \subset I^+$, and the third inequality follows from (iii). Since $|J^-| \le |I^-| \le 2|J^-|$, the last two computations imply that $\dashint_{I^-} d_E^{\,-\beta} \lesssim \dashint_{J^-} d_E^{\,-\beta}$ and $\dashint_{J^+} d_E^{\,-\alpha} \lesssim \dashint_{I^+} d_E^{\,-\alpha}$. Together with \eqref{eq_exp4b}, these two inequalities give \eqref{eq_exp4a}.
	
	\underline{Case 3}: Suppose $I^- \cap E = \emptyset$.
	
	\underline{Subcase 3a}: Suppose $d(I^+,E) \le 2|I^+|$. By Lemma \ref{lem_dist2}, we have $\dashint_{I^-} d_E^{\,-\beta} \lesssim_\beta |I^-|^{-\beta}$ and $\dashint_{I^+} d_E^{\,-\alpha} \gtrsim_\alpha |I^+|^{-\alpha}$, so \eqref{eq_exp4a} holds.
	
	\underline{Subcase 3b}: Suppose $d(I^+,E) \ge 2|I^+|$. Since $|d(I^-,E)-d(I^+,E)| \le \frac{1}{2}|I|$, we have $d(I^-,E) \ge |I^-|$. By Lemma \ref{lem_dist2}, $\dashint_{I^-} d_E^{\,-\beta} \approx_\beta d(I^-,E)^{-\beta}$ and $\dashint_{I^+} d_E^{\,-\alpha} \approx_\alpha d(I^+,E)^{-\alpha}$. Moreover, $d(I^-,E) \ge \frac{1}{2} d(I^+,E)$, so $d(I^-,E)^{-\beta} \lesssim d(I^+,E)^{-\beta}$. Hence \eqref{eq_exp4a} holds.
\end{proof}

We end this section with two corollaries of Theorems \ref{thm_exp1} and \ref{thm_exp2}.

\begin{corollary}
	\label{cor_exp1}
	Let $\emptyset \ne E \subset \R$.
	\begin{enumerate}[label=\upshape(\alph*)]
		\item $E$ is right weakly porous if and only if $\Mu_1^+(E) > 0$.
		\item $E$ is right median porous if and only if $\Mu_\infty^+(E) > 0$.
	\end{enumerate}
\end{corollary}

\begin{proof}
	(a) By Theorem \ref{thm_AGGM}, $E$ is right weakly porous if and only if $d_E^{\,-\alpha} \in A_1^+$ for some $\alpha>0$. By Theorem \ref{thm_exp1}, the latter statement holds if and only if $\Mu_1^+(E)>0$.
	
	(b) By Theorem \ref{thm_main}, $E$ is right median porous if and only if $d_E^{\,-\alpha} \in A_p^+$ for some $\alpha>0$ and $1<p<\infty$. By Theorem \ref{thm_exp2}, the latter statement holds if and only if $\Mu_\infty^+(E)>0$.
\end{proof}

\begin{corollary}
	\label{cor_exp2}
	Let $\emptyset \ne E \subset \R$.
	\begin{enumerate}[label=\upshape(\alph*)]
		\item $\Mu_1^+(E) \le \Mu_\infty^+(E)$.
		\item $\Mu_1(E) = \min(\Mu_1^+(E),\Mu_1^-(E))$.
		\item $\Mu_\infty(E) = \min(\Mu_\infty^+(E),\Mu_\infty^-(E))$.
	\end{enumerate}
\end{corollary}

\begin{proof}
	(a) Let $\alpha > 0$. Suppose $\alpha < \Mu_1^+(E)$. By Theorem \ref{thm_exp1}, $d_E^{\,-\alpha} \in A_1^+$. By Proposition \ref{prop_ApBMO}(a), $d_E^{\,-\alpha} \in A_p^+$ for all $1<p<\infty$. By Theorem \ref{thm_exp2}, $\alpha < \Mu_\infty^+(E)$. The desired inequality follows.
	
	(b) Let $\alpha > 0$. By Theorem \ref{thm_ALMV2}, we have $\alpha < \Mu_1(E)$ if and only if $d_E^{\,-\alpha} \in A_1$. By Proposition \ref{prop_ApBMO}(c), the latter statement holds if and only if $d_E^{\,-\alpha} \in A_1^+$ and $d_E^{\,-\alpha} \in A_1^-$. By Theorem \ref{thm_exp1} (and its obvious analogue for $A_1^-$ and $\Mu_1^-(E)$), this is equivalent to requiring that $\alpha < \Mu_1^+(E)$ and $\alpha < \Mu_1^-(E)$, i.e.\ $\alpha < \min(\Mu_1^+(E),\Mu_1^-(E))$. The desired equality follows.
	
	(c) This is similar to (b), but we use Theorems \ref{thm_PU2} and \ref{thm_exp2} instead of Theorems \ref{thm_ALMV2} and \ref{thm_exp1}.
\end{proof}

\section{Further results}
\label{sec_extra}

In this section, we prove a few additional results related to the notions of porosity considered in this paper. The first of these results (Proposition \ref{prop_extra1}) states that the property of right median porosity may be formulated in terms of a division of the interval $I$ into two subintervals which need not be of equal length (cf. Proposition \ref{prop_Ap}).

\begin{lemma}
	\label{lem_extra1}
	Suppose $E\subset\R$ is $(s,t,\delta)$-right median porous. Then, for every $r$ with $s \le r \le t$, there exists $\gamma>0$ such that, for all $I\subset\R$,
	\[
	L_r(E,I^-) \ge \gamma L_r(E,I).
	\]
\end{lemma}

\begin{proof}
	By Lemmas \ref{lem_length1}(a) and \ref{lem_porous1}(a), we may assume $E$ is closed. By Lemma \ref{lem_porous1}(b), $|E|=0$ and $L_s(I^-) \ge \delta L_t(I^+)$ for all $I\subset\R$. Let $I\subset\R$ be given. Write $I = (a,b)$ and let $c$ be the midpoint of $I$. By Lemma \ref{lem_length1}(b), $L_r(I)>0$ and there exist disjoint $I_1,\dots,I_n \subset I \setminus E$ such that $|I_i| \ge L_r(I)$ for $i=1,\dots,n$ and $\sum_{i=1}^n |I_i| \ge (1-r)|I|$. For $i = 1,\dots,n$, write $I_i = (a_i,b_i)$ and let $c_i$ be the midpoint of $I_i$. We may assume $c_1 < \cdots < c_n$. Let $\ell = \sum_{i=1}^n |I_i|$. Since $0<r<1$, we have $\ell>0$, so there is a unique $k \in \{1,\dots,n\}$ such that
	\[
	\sum_{i=1}^{k-1} |I_i| \le \frac{\ell}{2}
	\qquad \text{and} \qquad
	\sum_{i=1}^k |I_i| > \frac{\ell}{2}.
	\]
	
	We shall consider a number of cases and show that $L_r(I^-) \ge \gamma L_r(I)$ for some constant $\gamma$ (depending only on $s$, $t$, $\delta$ and $r$) in each case. By taking the minimum of these constants, one may obtain a constant which works in all cases.
	
	If $I_k \subset I^-$, then at least $1-r$ of $I^-$ is covered by $E$-free subintervals with length at least $L_r(I)$ (namely, $I_1,\dots,I_k$), so $L_r(I^-) \ge L_r(I)$.
	
	If $I_k \subset I^+$, then at least $1-r$ of $I^+$ is covered by $E$-free subintervals with length at least $L_r(I)$ (namely, $I_k,\dots,I_n$), so $L_r(I^+) \ge L_r(I)$. Since $L_r(I^-) \ge L_s(I^-) \ge \delta L_t(I^+) \ge \delta L_r(I^+)$, this implies that $L_r(I^-) \ge \delta L_r(I)$.
	
	For the rest of the proof, suppose $I_k \not\subset I^-$ and $I_k \not\subset I^+$. Then $c \in I_k$. If $\sum_{i=1}^{k-1} |I_i| = \frac{\ell}{2}$, then at least $1-r$ of $I^-$ is covered by $E$-free subintervals with length at least $L_r(I)$ (namely, $I_1,\dots,I_{k-1}$), so $L_r(I^-) \ge L_r(I)$. For the rest of the proof, suppose $\sum_{i=1}^{k-1} |I_i| < \frac{\ell}{2}$. Then there is a unique $x \in I_k$ such that
	\[
	\sum_{i=1}^{k-1} |I_i| + (x-a_k) = \frac{\ell}{2} = (b_k-x) + \sum_{i=k+1}^n |I_i|.
	\]
	
	If $c \ge c_k$ and $c \ge x$, then at least $\frac{1}{2}$ of $I_k$ and at least $\frac{1}{2}$ of the total length of $I_1,\dots,I_n$ lie in $I^-$, so at least $1-r$ of $I^-$ is covered by $E$-free subintervals with length at least $\frac{1}{2}L_r(I)$ (namely, $I_1,\dots,I_{k-1},I_k \cap I^-$) and hence $L_r(I^-) \ge \frac{1}{2}L_r(I)$.
	
	If $c \le c_k$ and $c \le x$, then at least $\frac{1}{2}$ of $I_k$ and at least $\frac{1}{2}$ of the total length of $I_1,\dots,I_n$ lie in $I^+$, so at least $1-r$ of $I^+$ is covered by $E$-free subintervals with length at least $\frac{1}{2}L_r(I)$ (namely, $I_k \cap I^+, I_{k+1},\dots,I_n$) and hence $L_r(I^+) \ge \frac{1}{2}L_r(I)$, which implies that $L_r(I^-) \ge \frac{\delta}{2}L_r(I)$.
	
	In the rest of the proof, we consider the remaining cases $x \le c \le c_k$ and $c_k \le c \le x$.
	
	\underline{Case 1}: Suppose $x \le c \le c_k$. Choose $\lambda>0$ small enough that $\lambda<\frac{\delta}{4}$ and $1/(1+\frac{3\lambda}{\delta}) > 1-r$. If $|I_k \cap I^-| \ge \lambda L_r(I)$, then at least $1-r$ of $I^-$ is covered by $E$-free subintervals with length at least $\lambda L_r(I)$ (namely, $I_1,\dots,I_{k-1},I_k \cap I^-$), so $L_r(I^-) \ge \lambda L_r(I)$. For the rest of this case, suppose $|I_k \cap I^-| \le \lambda L_r(I)$. Let $J_1=(a,a_1)$, $J_i = (b_{i-1},a_i)$ for $i=2,\dots,n$, and $J_{n+1}=(b_n,b)$. (Some of the intervals $J_i$ may be empty.) Then $J_1 \sqcup I_1 \sqcup \cdots \sqcup J_n \sqcup I_n \sqcup J_{n+1} = I$ and
	\[
	J_1 \sqcup I_1 \sqcup \cdots \sqcup J_{k-1} \sqcup I_{k-1} \sqcup J_k \sqcup (I_k \cap I^-) = I^-.
	\]
	
	\underline{Subcase 1a}: Suppose $|J_i| \ge \frac{\lambda}{\delta}L_r(I)$ for some $i \in \{1,\dots,k\}$. Note that $|I_i| \ge \frac{\lambda}{\delta}L_r(I)$. Let $H$ be the unique interval such that $H^- \subset J_i$, $H^+ \subset I_i$, and $|H^-|=|H^+|=\frac{\lambda}{\delta}L_r(I)$. Then $H^+$ is $E$-free, so $L_t(H^+) = |H^+|$ and hence $L_s(H^-) \ge \delta L_t(H^+) = \lambda L_r(I)$. By Lemma \ref{lem_length1}(b), $H^-$ contains an $E$-free interval $K$ with length at least $L_s(H^-)$. Now, $K \subset J_i$ and $|K| \ge \lambda L_r(I) \ge x-a_k$, so at least $1-r$ of $I^-$ is covered by $E$-free subintervals with length at least $\lambda L_r(I)$ (namely, $I_1,\dots,I_{k-1},K$). Hence $L_r(I^-) \ge \lambda L_r(I)$.
	
	\underline{Subcase 1b}: Suppose $|J_i| \le \frac{\lambda}{\delta}L_r(I)$ for all $i \in \{1,\dots,k\}$. If $k = 1$, then
	\[
	|I^-| = |J_1| + |I_1 \cap I^-| \le \tfrac{\lambda}{\delta}L_r(I) + \lambda L_r(I) \le \tfrac{2\lambda}{\delta}|I| < \tfrac{1}{2}|I|,
	\]
	which is a contradiction. Thus, $k \ge 2$. Let $A = \sum_{i=1}^{k-1} |I_i|$ and $B = \sum_{i=1}^k |J_i| + |I_k \cap I^-|$. Then $A+B = |I^-|$, $A \ge (k-1)L_r(I)$, and $B \le k \frac{\lambda}{\delta} L_r(I) + \lambda L_r(I) \le (k+1) \frac{\lambda}{\delta} L_r(I)$, so
	\begin{align*}
		\frac{A}{|I^-|}
		&= \frac{A}{A+B}
		\ge \frac{(k-1)L_r(I)}{(k-1)L_r(I)+B}
		\ge \frac{(k-1)L_r(I)}{(k-1)L_r(I)+(k+1)\frac{\lambda}{\delta}L_r(I)} \\
		&= \frac{1}{1+\frac{k+1}{k-1} \cdot \frac{\lambda}{\delta}}
		\ge \frac{1}{1+\frac{3\lambda}{\delta}}
		> 1-r.
	\end{align*}
	Thus, at least $1-r$ of $I^-$ is covered by $E$-free subintervals with length at least $L_r(I)$ (namely, $I_1,\dots,I_{k-1}$), so $L_r(I^-) \ge L_r(I)$.
	
	\underline{Case 2}: Suppose $c_k \le c \le x$.
	
	\underline{Subcase 2a}: Suppose $r>s$. Let $\epsilon = \frac{1}{6}(r-s)|I_k|$ and $\lambda = \frac{1}{6}(r-s)\delta$. Let $a' = a - 2\epsilon$ and $c' = c - \epsilon$. Since $\epsilon < \frac{1}{2}|I_k|$, we have $c' \in I_k$. Let $J$ be the interval such that $J^- = (a',c')$ and $J^+ = (c',b)$. Then $(c',c) \subset I_k \cap J^+$ and hence $|I_k \cap J^+| \ge \epsilon \ge \frac{1}{6}(r-s)L_r(I)$, so at least $1-r$ of $J^+$ is covered by $E$-free subintervals with length at least $\frac{1}{6}(r-s)L_r(I)$ (namely, $I_k \cap J^+,I_{k+1},\dots,I_n$). Thus, $L_r(J^+) \ge \frac{1}{6}(r-s)L_r(I)$, so $L_s(J^-) \ge \lambda L_r(I)$.
	
	Let $\ell_0$ be the total length of $I^-$ covered by $E$-free subintervals with length at least $\lambda L_r(I)$. By Lemma \ref{lem_length1}(b), at least $1-s$ of $J^-$ is covered by $E$-free subintervals with length at least $L_s(J^-)$. Since $L_s(J^-) \ge \lambda L_r(I)$, this implies that at least $1-s$ of $J^-$ is covered by $E$-free subintervals with length at least $\lambda L_r(I)$. Observe that, if $K = (\alpha,\beta)$ is an $E$-free subinterval of $J^-$ with length at least $\lambda L_r(I)$, and if $\beta \ge a + \lambda L_r(I)$, then $K \cap I^-$ is an $E$-free subinterval of $I^-$ with length at least $\lambda L_r(I)$. Using this observation, we estimate
	\begin{align*}
		\ell_0
		&\ge (1-s)|J^-| - (2\epsilon + \lambda L_r(I)) \\
		&= (1-s)|I^-| - (1+s)\epsilon - \lambda L_r(I) \\
		&\ge (1-s)|I^-| - 2\epsilon - \tfrac{\lambda}{\delta} |I_k| \\
		&= (1-s)|I^-| - \tfrac{1}{2}(r-s)|I_k| \\
		&\ge (1-r)|I^-|.
	\end{align*}
	Therefore, $L_r(I^-) \ge \lambda L_r(I)$.
	
	\underline{Subcase 2b}: Suppose $r<t$. Let $\mu = \frac{1}{2}(t-r)$. If $|I_k \cap I^+| \ge \mu L_r(I)$, then at least $1-r$ of $I^+$ is covered by $E$-free subintervals with length at least $\mu L_r(I)$ (namely $I_k \cap I^+,I_{k+1},\dots,I_n$), so $L_r(I^+) \ge \mu L_r(I)$ and hence $L_r(I^-) \ge \delta \mu L_r(I)$. Now, suppose $|I_k \cap I^+| \le \mu L_r(I)$. Let $A = |I_k \cap I^+|$ and $B = \sum_{i=k+1}^n |I_i|$. Then $A \le \mu|I|$ and $A+B \ge \frac{1}{2}(1-r)|I|$, so $B \ge \frac{1}{2}(1-t)|I|$. Thus, at least $1-t$ of $I^+$ is covered by $E$-free subintervals with length at least $L_r(I)$ (namely, $I_{k+1},\dots,I_n$), so $L_t(I^+) \ge L_r(I)$ and hence $L_r(I^-) \ge \delta L_r(I)$.
\end{proof}

\begin{proposition}
	\label{prop_extra1}
	Suppose $E\subset\R$ is $(s,t,\delta)$-right median porous. Then, for every $\sigma$ with $s<\sigma<t$ and every $C\ge1$, there exists $\gamma>0$ such that, for all $a<b<c$ with $\frac{1}{C} \le \frac{c-b}{b-a} \le C$, we have
	\[
	L_\sigma(E,(a,b)) \ge \gamma L_t(E,(b,c)).
	\]
\end{proposition}

\begin{proof}
	By Lemmas \ref{lem_length1}(a) and \ref{lem_porous1}(a), we may assume $E$ is closed. By Lemma \ref{lem_porous1}(b), $|E|=0$ and $L_s(I^-) \ge \delta L_t(I^+)$ for all $I\subset\R$. By Lemma \ref{lem_extra1}, there exists $0<\epsilon<1$ such that $L_t(I^-) \ge \epsilon L_t(I)$ for all $I\subset\R$. Choose $k \in \Z_{>0}$ large enough that $k>C$ and $\frac{k}{k+1}>\frac{1-\sigma}{1-s}$, and let $\gamma = \delta^{2k} \epsilon^{1 + \log_2(kC)}$.
	
	Let $a<b<c$ with $\frac{1}{C} \le \frac{c-b}{b-a} \le C$ be given. Then $c-b > \frac{1}{k}(b-a)$, so there is a unique $n \in \Z_{>0}$ such that
	\begin{equation}
		\label{eq_extra1}
		\frac{1}{2k}(b-a) < 2^{-n}(c-b) \le \frac{1}{k}(b-a).
	\end{equation}
	Note that $2^{n-1} < k \cdot \frac{c-b}{b-a} \le kC$ and hence $n < 1 + \log_2(kC)$. Let $I_0 = (b,x)$, where $x-b = 2^{-n}(c-b)$. By applying the inequality $L_t(I^-) \ge \epsilon L_t(I)$, $n$ times, we get
	\begin{equation}
		\label{eq_extra2}
		L_t(I_0) \ge \epsilon^n L_t((b,c)).
	\end{equation}
	For each $i \in \Z_{>0}$, let $I_i = J_i^-$, where $J_i$ is the unique interval such that $J_i^+ = I_{i-1}$. Let $m$ be the largest index such that $I_m \subset (a,b)$. Since $|I_i| = 2^{-n}(c-b)$ for all $i \ge 0$, \eqref{eq_extra1} implies that $k \le m < 2k$. By right median porosity, we have $L_s(I_i) \ge \delta L_t(I_{i-1})$ for all $i \ge 1$. Since $L_t(I_i) \ge L_s(I_i)$, this inequality may be iterated to get, for $i=1,\dots,m$,
	\begin{equation}
		\label{eq_extra3}
		L_s(I_i) \ge \delta^i L_t(I_0) \ge \delta^m L_t(I_0).
	\end{equation}
	Combining \eqref{eq_extra2} and \eqref{eq_extra3}, we get, for $i = 1,\dots,m$,
	\[
	L_s(I_i) \ge \delta^m \epsilon^n L_t((b,c)) \ge \gamma L_t((b,c)).
	\]
	By Lemma \ref{lem_length1}(b), at least $1-s$ of $I_i$ is covered by $E$-free subintervals with length at least $L_s(I_i)$. Hence at least $1-s$ of $I_1 \sqcup \cdots \sqcup I_m$ is covered by $E$-free subintervals with length at least $\gamma L_t((b,c))$. By the maximality of $m$, we have $m|I_0| \le b-a < (m+1)|I_0|$, so
	\[
	\frac{|I_1 \sqcup \cdots \sqcup I_m|}{b-a}
	> \frac{m|I_0|}{(m+1)|I_0|}
	\ge \frac{k}{k+1}
	> \frac{1-\sigma}{1-s}.
	\]
	Thus, at least $1-\sigma$ of $(a,b)$ is covered by $E$-free subintervals with length at least $\gamma L_t((b,c))$, so $L_\sigma((a,b)) \ge \gamma L_t((b,c))$, as desired.
\end{proof}

Our final result in this paper (Proposition \ref{prop_extra2}) helps to illuminate the relationship between the notions of porosity considered in this paper and their dyadic analogues. For each notion of porosity listed in Definition \ref{def_porous}, one may formulate a dyadic analogue in which, instead of arbitrary subintervals of $I$, one considers only dyadic subintervals of $I$. In fact, weak porosity and median porosity are defined in this way in \cite{ALMV} and \cite{PU}, respectively. Using Proposition \ref{prop_extra2} and Lemma \ref{lem_porous1}(b) (and its analogues for the other notions of porosity), it is easy to see that, for each notion of porosity, the definition with dyadic subintervals is equivalent to the definition with arbitrary subintervals, if we allow changes in the values of $s$, $t$ and $\delta$. Furthermore, the changes in $s$ and $t$ may be as small as desired.

\begin{proposition}
	\label{prop_extra2}
	For any $E \subset \R$, $0 < s \le 1$, and $I \subset \R$, let $L_s^d(E,I)$ be defined in the same way as $L_s(E,I)$, except that $I_1,\dots,I_k$ are now required to be dyadic subintervals of $I$. Then the following hold:
	\begin{enumerate}[label=\upshape(\alph*)]
		\item $L_1^d(E,I) \le L_1(E,I) \le 4 L_1^d(E,I)$.
		\item If $0<s<t<1$, then $L_s^d(E,I) \le L_s(E,I) \lesssim_{s,t} L_t^d(E,I)$.
	\end{enumerate}
\end{proposition}

\begin{proof}
	(a) The first inequality is trivial since $L_1^d(E,I)$ is defined as the supremum of a smaller collection than $L_1(E,I)$. To prove the second inequality, assume $L_1(E,I) > 0$ and let $0 < \ell < L_1(E,I)$. Then there exists $J \subset I \setminus E$ such that $\ell < |J| \le L_1(E,I)$. Since $0 < |J| \le |I|$, there is a unique $n \in \Z_{>0}$ such that $2^{-n}|I| < |J| \le 2^{1-n}|I|$. Since $2^{-n-1}|I| < \frac{1}{2}|J|$, there exists $K \in \D_{n+1}(I)$ such that $K \subset J$. Then $K \subset I \setminus E$, so $L_1^d(E,I) \ge |K| = 2^{-n-1}|I| > \frac{1}{4}\ell$. Letting $\ell \to L_1(E,I)$, we get the desired inequality.
	
	(b) As in (a), the first inequality is trivial. To prove the second inequality, assume $L_s(E,I) > 0$ and let $0 < \ell < L_s(E,I)$. Then there exist disjoint $I_1,\dots,I_k \subset I \setminus E$ such that $|I_i| \ge \ell$ for $i = 1,\dots,k$ and $\sum_{i=1}^k |I_i| \ge (1-s)|I|$. Since $0 < \ell < |I|$, there is a unique $n \in \Z_{>0}$ such that $2^{-n}|I| < \ell \le 2^{1-n}|I|$.
	
	Choose $m \in \Z_{>0}$ large enough that $1-2^{1-m} > \frac{1-t}{1-s}$. Then, for $i = 1,\dots,k$, we have
	\[
	\sum_{\substack{J \in \D_{n+m}(I) \\ J \subset I_i}} |J|
	\ge |I_i| - 2 \cdot 2^{-n-m} |I|
	> (1-2^{1-m}) |I_i|.
	\]
	Hence
	\[
	\sum_{\substack{J \in \D_{n+m}(I) \\ J \subset \bigsqcup_{i=1}^k I_i}} |J|
	\ge (1-2^{1-m}) \sum_{i=1}^k |I_i|
	> (1-t)|I|.
	\]
	The intervals $J \in \D_{n+m}(I)$ such that $J \subset \bigsqcup_{i=1}^k I_i$ are disjoint $E$-free dyadic subintervals of $I$, each of length $2^{-n-m}|I|$, and we have shown that they cover more than $1-t$ of $I$, so $L_t^d(E,I) \ge 2^{-n-m}|I| \ge 2^{-m-1}\ell$. Letting $\ell \to L_s(E,I)$, we get $L_t^d(E,I) \ge 2^{-m-1} L_s(E,I)$.
\end{proof}

\section{Acknowledgements}

A.\ C.\ Goksan was partially supported by an NSERC Canada Graduate Scholarship. I.\ Uriarte-Tuero was partially supported by an NSERC Discovery Grant.

\bibliographystyle{plain}
\bibliography{bibliography}

\end{document}